\documentclass{amsart}

\usepackage{graphicx,mathrsfs,amssymb,amsmath,amsthm,amscd}

\newtheorem{theorem}{Theorem}[section]
\newtheorem{lemma}[theorem]{Lemma}

\newtheorem{proposition}[theorem]{Proposition}
\newtheorem{remark}[theorem]{Remark}
\newtheorem{definition}[theorem]{Definition}
\newtheorem{example}[theorem]{Example}

\numberwithin{equation}{section}

\newcommand{\cz}{{\mathbb C}}
\newcommand{\gz}{{\mathbb Z}}
\newcommand{\nz}{{\mathbb N}}
\newcommand{\rz}{{\mathbb R}}

\newcommand{\bfS}{\mathbf{S}}
\newcommand{\bfT}{\mathbf{T}}
\newcommand{\bfH}{\mathbf{H}}

\newcommand{\calA}{\mathcal{A}}
\newcommand{\calB}{\mathcal{B}}

\newcommand{\calE}{\mathcal{E}}

\newcommand{\fraka}{\mathfrak{a}}
\newcommand{\frakA}{\mathfrak{A}}
\newcommand{\frakB}{\mathfrak{B}}

\newcommand{\frakC}{\mathfrak{C}}
\newcommand{\frakd}{\mathfrak{d}}
\newcommand{\frakD}{\mathfrak{D}}
\newcommand{\frakg}{\mathfrak{g}}

\newcommand{\frakP}{\mathfrak{P}}
\newcommand{\frakQ}{\mathfrak{Q}}
\newcommand{\frakS}{\mathfrak{S}}

\newcommand{\scrA}{\mathscr{A}}
\newcommand{\scrB}{\mathscr{B}}
\newcommand{\scrC}{\mathscr{C}}

\newcommand{\scrH}{\mathscr{H}}

\newcommand{\scrL}{\mathscr{L}}
\newcommand{\scrP}{\mathscr{P}}
\newcommand{\scrR}{\mathscr{R}}
\newcommand{\scrS}{\mathscr{S}}

\newcommand{\cl}{\mathrm{cl}}
\newcommand{\dbar}{\overline{\partial}}
\newcommand{\dzbar}{d\overline{z}}

\newcommand{\forget}[1]{}
\newcommand{\im}{\text{\rm Im}\,}
\newcommand{\lra}{\longrightarrow}

\newcommand{\pO}{\partial\Omega}
\newcommand{\re}{\mathrm{Re}\,}

\newcommand{\wh}{\widehat}
\newcommand{\wt}{\widetilde}

\begin{document}
\title[Elliptic complexes on manifolds with boundary]%
{Elliptic complexes on\\ manifolds with boundary} 

\author{B.-W. Schulze}
\address{Universit\"at Potsdam, Institut f\"ur Mathematik, Potsdam (Germany)}
\email{schulze@math.uni-potsdam.de}

\author{J.\ Seiler}
\address{Universit\`a di Torino, Dipartimento di Matematica, Torino (Italy)}
\email{joerg.seiler@unito.it}


\begin{abstract}
We show that elliptic complexes of  $($pseudo$)$differential operators on smooth compact manifolds with boundary 
can always be complemented to a Fredholm problem by boundary conditions involving global pseudodifferential 
projections on the boundary $($similarly as the spectral boundary conditions of Atiyah, Patodi and Singer for 
a single operator$)$. We prove that boundary conditions without projections can be chosen if, and only if, 
the topological Atiyah-Bott obstruction vanishes. These results make use of a Fredholm theory for complexes 
of operators in algebras of generalized  pseudodifferential operators of  Toeplitz type which we also develop in the 
present paper. \end{abstract}

\vspace*{-20mm}

\maketitle



{\small
\tableofcontents
}
\section{Introduction}\label{sec:intro}

The present paper is concerned with the Fredholm theory of complexes of differential operators and, more generally, 
of complexes of operators belonging to pseudodifferential operator algebras.  
In particular, we consider complexes of differential operators on 
manifolds with boundary and investigate the question in which way one can complement complexes, 
which are elliptic on the level of homogeneous principal symbols, with boundary conditions to achieve a Fredholm problem. 
A boundary condition means here a homomorphism between the given complex and a complex of pseudodifferential operators 
on the boundary; it is called a Fredholm problem if the associated mapping cone has finite-dimensional cohomology spaces  
(see Sections \ref{sec:02.2} and \ref{sec:03.2} for details). 
As we shall show, boundary conditions can always be found, but the character of the boundary conditions to be 
chosen depends on the presence of a topological obstruction, the so-called Atiyah-Bott obstruction, 
cf.\ Atiyah and Bott \cite{Atiy11}, here formulated for complexes. 
In case this obstruction vanishes, one may take ``standard'' conditions $($to be explained below$)$, otherwise 
one is lead to conditions named generalized Atiyah-Patodi-Singer conditions, since they involve global pseudodifferential 
projections on the boundary, similar as the classical spectral boundary conditions of Atiyah, Patodi and Singer \cite{Atiy8} 
for a single operator. Moreover, given a complex together with such kind of boundary conditions, we show that its Fredholm 
property is characterized by the exactness of two associated families of complexes being made up from the homogeneous 
principal symbols and the so-called homogeneous boundary symbols, respectively.  

Essential tools in our approach are a systematic use of  Boutet de Monvel's calculus $($or ``algebra''$)$ 
for boundary value problems \cite{Bout} $($see also Grubb \cite{Grub}, Rempel and Schulze \cite{ReSc}, and 
Schrohe \cite{Schr}$)$ and a suitable extension of it due to the first author \cite{Schu37}, as well as the concept 
of generalized pseudodifferential operator algebras of Toeplitz type in the spirit of the second author's work \cite{Seil}. 
A key role will play the results obtained in Sections \ref{sec:05} and \ref{sec:06} concerning complexes of such 
Toeplitz type operators. Roughly speaking, in these two sections we show how to construct an elliptic theory for complexes 
of operators belonging to an operator-algebra having a notion of ellipticity, and then how this theory can be lifted to
complexes involving projections from the algebra.  
We want to point out that these results do not only apply 
to complexes of operators on manifolds with boundary, but to complexes of operators belonging to any ``reasonable'' 
pseudodifferential calculus including, for example,  the calculi of the first author for manifolds with cone-, edge- and higher 
singularities \cite{Schu91} and Melrose's $b$-calculus for manifolds with corners \cite{Melr93}. 

Boutet de Monvel's calculus was designed for admitting the construction of parametrices 
$($i.e., inverses modulo ``smoothing'' or ``regularizing'' operators$)$ of Shapiro-Lopatinskij elliptic boundary value problems on a 
manifold $\Omega$ within an optimal pseudodifferential setting. The elements of this algebra are $2\times2$ block-matrix operators acting 
between smooth or Sobolev sections of vector bundles over $\Omega$ and its boundary $\pO$, respectively; see Section \ref{sec:03.1} 
for further details. Boutet de Monvel also used his calculus to prove an analogue of the Atiyah-Singer index theorem in $K$-theoretic terms. 
There arises the question whether any given elliptic differential operator $A$ on $\Omega$ $($i.e., $A$ has an invertible homogeneous  principal symbol$)$ can be complemented by boundary conditions to yield an elliptic boundary value problem belonging to Boutet de Monvel's calculus. The answer is no, in general. In fact, the so-called Atiyah-Bott obstruction must vanish for $A$: Specifying a normal coordinate near the boundary, one can associate with $A$ its boundary symbol $\sigma_\partial(A)$ which is defined on the unit co-sphere bundle $S^*\pO$ of the boundary $\pO$ and takes values in the differential operators on the half-axis $\rz_+$. In case of ellipticity, this is a family of Fredholm operators between suitable Sobolev spaces of the half-axis, hence generates an element of the $K$-group of $S^*\pO$, the so-called index element. We shall denote this index element by $\mathrm{ind}_{S^*\pO}\,\sigma_\partial(A)$. The Atiyah-Bott obstruction asks that the index element belongs to $\pi^*K(\pO)$, the pull-back of the $K$-group of $\pO$ under the canonical projection $\pi:S^*\pO\to\pO$. 

A simple example of an operator violating the Atiyah-Bott obstruction is the Cauchy-Riemann operator $\dbar$ on the unit-disc $\Omega$ in $\rz^2$, see Section \ref{sec:03.1.5} for more details. However, in this case we may substitute the Dirichlet condition $u\mapsto\gamma_0u$ by $u\mapsto C\gamma_0u$, where $C$ is the associated Calder\'{o}n projector, which is a zero order pseudodifferential projection on the boundary. One obtains Fredholm operators $($in fact, invertible operators$)$, say from $H^s(\Omega)$ to $H^{s-1}(\Omega)\oplus H^{s-1/2}(\pO;C)$, where $H^s(\pO;C)$ denotes the range space of $C$. In \cite{Seel}, Seeley has shown that this works for every elliptic differential operator on a smooth manifold. Schulze, Sternin and Shatalov in \cite{Schu24} considered boundary value problems for elliptic differential operators $A$ with boundary conditions of the form $u\mapsto PB\gamma u$, where $\gamma$ is the operator mapping $u$ to the vector of its first $\mu-1$ derivatives $\partial^j_\nu u|_{\pO}$ in normal direction, $B$ and $P$ are pseudodifferential operators on the boundary and $P$ is a zero-order projection. They showed that the Fredholm property of the resulting operator, where $PB\gamma$ is considered as a map into the image of $P$ rather than into the full function spaces over the boundary, can be characterized by the invertibility of suitably associated principal symbols. Based on these results, the first author of the present work has constructed in \cite{Schu37} a pseudodifferential calculus containing such  boundary value problems, extending Boutet de Monvel's calculus. This calculus permits to construct parametrices of elliptic elements, where the notion of ellipticity is now defined in a new way, taking into account the presence of the projections; see Section \ref{sec:03.1.2} for details.  In \cite{SS04} the authors realized this concept for boundary value problems without the transmission property and in \cite{SS06} they consider operators on manifolds with edges. 

While \cite{Schu37}, \cite{SS04} and \cite{SS06} exclusively dealt with the question of how to incorporate global projection conditions in a specific pseudodifferential calculus $($Boutet de Monvel's calculus and Schulze's algebra of edge pseudodifferential operators, respectively$)$, the second author in \cite{Seil} considered this question from a more general point of view: Given a calculus of  ``generalized''  pseudodifferential operators $($see Section \ref{sec:04.1} for details$)$ with a notion of ellipticity and being closed under construction of parametrices, how can one build up a wider calculus containing all \emph{Toeplitz type} operators $P_1AP_0$, where $A$, $P_0$, $P_1$ belong to the original calculus and the $P_j=P_j^2$ are projections? It turns out that if the original calculus has some natural key properties, then the notion of ellipticity and the parametrix construction extend in a canonical way to the class of Toeplitz type operators; see Section \ref{sec:04.2} for details. 

In the present paper we are not concerned with single operators but with complexes of operators. There is no need to emphasize the importance of  operator complexes in mathematics and that they have been studied intensely in the past, both in concrete $($pseudo-$)$differential and more abstract settings; let us only mention the works of Ambrozie and Vasilescu \cite{AmVa},  Atiyah and Bott \cite{Atiy4}, Br\"uning and Lesch \cite{BrLe}, Rempel and Schulze \cite{ReSc} and Segal \cite{Sega1}, \cite{Sega2}. The Fredholm property of a single operator is now replaced by the Fredholm property of the complex, i.e., the property of having finite-dimensional cohomology spaces. In Section \ref{sec:02} we shortly summarize some basic facts on complexes of operators in Hilbert spaces and use the occasion to correct an erroneous statement present in the literature concerning the Fredholm property of mapping cones, cf.\ Proposition \ref{prop:mapping_cone} and the example given before.  

A complex of differential operators on a manifold with boundary which is exact $($respectively, acyclic$)$ on the level of homogeneous principal symbols, in general will not have the Fredholm property. Again it is natural to ask whether it is possible to complement the complex with boundary conditions to a Fredholm problem within the framework of Boutet de Monvel's calculus. Already Dynin, in his two-page note \cite{Dyni}, observed the presence of a kind of Atiyah-Bott obstruction which singles out those complexes that can be complemented with trace operators from Boutet de Monvel's calculus. Unfortunately, \cite{Dyni} does not contain any proofs and main results claimed there could not be reproduced later on. One contribution of our paper is to construct complementing boundary conditions  in case of vanishing Atiyah-Bott obstruction, though of a different form as those announced in \cite{Dyni}. Moreover we show that, in case of violated Atiyah-Bott obstruction, we can complement the complex with generalized Atiyah-Patodi-Singer conditions to a Fredholm complex, see Section \ref{sec:03.2}. Given a complex with boundary conditions, we characterize its Fredholm property on principal symbolic level. 

As is well-known, for the classical deRham complex on a bounded manifold the Atiyah-Bott obstruction vanishes; in fact, the complex itself -- without any additional boundary condition -- is a Fredholm complex. On the other hand, the Dolbeault or Cauchy-Riemann complex on a complex manifold with boundary violates  the Atiyah-Bott obstruction; we shall show this in Section \ref{sec:03.3.5} in the simple case of the two-dimensional unit ball, where calculations are very explicit. Still, by our result, the Dolbeault complex can be complemented by generalized Atiyah-Patodi-Singer conditions to a Fredholm problem. 

\section{Complexes in Hilbert spaces}\label{sec:02}

In this section we shall provide some basic material about complexes of bounded operators and shall introduce 
some notation that will be used throughout this paper.  

\subsection{Fredholm complexes and parametrices}\label{sec:02.1}

A \emph{Hilbert space complex} consists of a family of Hilbert spaces $H_j$, $j\in\gz$, together with a family 
of operators $A_j\in\scrL(H_j,H_{j+1})$ satisfying $A_{j+1}A_j=0$ for any $j$ $($or, equivalently, 
$\mathrm{im}\,A_j\subseteq\mathrm{ker}\,A_{j+1}$ for any $j)$. More intuitively, we shall represent a complex 
as a diagram  
 $$\frakA:\ldots\lra H_{-1}\xrightarrow{A_{-1}}
     H_0\xrightarrow{A_0}H_1\xrightarrow{A_1}H_2
     \xrightarrow{A_2}H_3\lra\ldots$$
Mainly we shall be interested in \emph{finite} complexes, i.e., the situation where $H_j=\{0\}$ for $j<0$ and 
$j>n+1$ for some natural number $n$. In this case we write 
 $$\frakA:0\lra H_0\xrightarrow{A_0}H_1\xrightarrow{A_1}\ldots\xrightarrow{A_{n-1}}
     H_n\xrightarrow{A_n}H_{n+1}\lra 0.$$

\begin{definition}
The cohomology spaces of the complex $\frakA$ are denoted by 
 $$\scrH_j(\frakA)=\mathrm{ker}\,A_j \big/ \mathrm{im}\,A_{j-1},\qquad j\in\gz.$$
\end{definition}

In case $\scrH_j(\frakA)$ is finite dimensional, the operator $A_{j-1}$ has closed range. We call $\frakA$ a 
\emph{Fredholm complex} if all cohomology spaces are of finite dimension. In case $\frakA$ is also finite, 
we then define the index of $\frakA$ as 
 $$\mathrm{ind}\,\frakA=\sum_{j}(-1)^j\mathrm{dim}\,\scrH_j(\frakA).$$
The complex $\frakA$ is called \emph{exact in position $j$}, if the $j$-th cohomology space is trivial; it 
is called \emph{exact} $($or also acyclic$)$ if it is exact in every position $j\in\gz$. 

\begin{definition}
The $j$-th Laplacian associated with $\frakA$ is the operator 
 $$\Delta_j:=A_{j-1}A_{j-1}^*+A_j^*A_j\in\scrL(H_j).$$
\end{definition}

In case $\mathrm{dim}\,\scrH_j(\frakA)<+\infty$, the orthogonal decomposition 
 $$\mathrm{ker}\,A_j=\mathrm{im}\,A_{j-1}\oplus \mathrm{ker}\,\Delta_j$$ 
is valid; in particular, we can write 
 $$H_j=(\mathrm{ker}\,A_j)^\perp\oplus \mathrm{im}\,A_{j-1}\oplus \mathrm{ker}\,\Delta_j,$$ 
and $\frakA$ is exact in position $j$ if, and only if, $\Delta_j$ is an isomorphism. 

\begin{definition}\label{def:paramterix}
A parametrix of $\frakA$ is a sequence of operators $P_j\in\scrL(H_{j+1},H_j)$, $j\in\gz$,  
such that the following operators are compact$:$
 $$A_{j-1}P_{j-1}+P_jA_j-1\in\scrL(H_j),\qquad j\in\gz.$$
\end{definition}

Note that in the definition of the parametrix we do {not} require that $P_jP_{j+1}=0$ for every $j$; in case 
this property is valid, we also call $\frakP$ a complex and represent it schematically as 
 $$\frakP:\ldots\longleftarrow H_{-1}\xleftarrow{P_{-1}}
     H_0\xleftarrow{P_0}H_1\xleftarrow{P_1}H_2
     \xleftarrow{P_2}H_3\longleftarrow\ldots$$

\begin{theorem}\label{thm:complex-basics}
For $\frakA$ the following properties are equivalent$:$
\begin{itemize}
 \item[a$)$] $\frakA$ is a Fredholm complex. 
 \item[b$)$] $\frakA$ has a parametrix. 
 \item[c$)$] $\frakA$ has a parametrix which is a complex. 
 \item[d$)$] All Laplacians $\Delta_j$, $j=0,1,2,\ldots$, are Fredholm operators in $H_j$. 
\end{itemize}
\end{theorem} 

\subsection{Morphisms and mapping cones}\label{sec:02.2}

Given two complexes $\frakA$ and $\frakQ$, a morphism $\bfT: \frakA\to\frakQ$ is a sequence of operators 
$T_j\in\scrL(H_j,L_j)$, $j\in\gz$, such that the following diagram is commutative$:$ 
$$
\begin{CD}
\ldots @>>> H_{-1} @>A_{-1}>> H_0 @>A_{0}>> H_1 @>A_{1}>> H_2  @>>> \ldots \\
@.  @VV{T_{-1}}V  @VV{T_{0}}V @VV{T_{1}}V @VV{T_{2}}V  \\ 
\ldots @>>> L_{-1} @>Q_{-1}>> L_0 @>Q_{0}>> L_1 @>Q_{1}>> L_2  @>>> \ldots 
\end{CD}
$$
i.e., $T_{j+1}A_j=Q_jT_j$ for every $j$. Note that these identities imply that $A_j(\mathrm{ker}\,T_j)\subseteq
\mathrm{ker}\,T_{j+1}$ and $Q_j(\mathrm{im}\,T_j)\subseteq\mathrm{im}\,T_{j+1}$ for every $j$. 

\begin{definition}\label{def:mapping-cone}
The \emph{mapping cone} associated with  $\bfT$ is the 
complex
 $$\frakC_{\bfT}:\;\ldots\lra 
     \begin{matrix}H_{-1}\\ \oplus\\ L_{-2}\end{matrix}
     \xrightarrow{\begin{pmatrix}-A_{-1} & 0\\ T_{-1} & Q_{-2}\end{pmatrix}}
     \begin{matrix}H_{0}\\ \oplus\\ L_{-1}\end{matrix}
     \xrightarrow{\begin{pmatrix}-A_{0} & 0\\ T_{0} & Q_{-1}\end{pmatrix}}
     \begin{matrix}H_{1}\\ \oplus\\ L_{0}\end{matrix}
     \lra\ldots$$
$\bfT$ is called a \emph{Fredholm morphism} if its mapping cone is a Fredholm complex. 
\end{definition}

We can associate with $\bfT$ two other complexes, namely 
 $$\mathrm{ker}\,\bfT:\ldots\lra \mathrm{ker}\,T_{-1}\xrightarrow{A_{-1}}
     \mathrm{ker}\,T_{0}\xrightarrow{A_0}
     \mathrm{ker}\,T_{1}\xrightarrow{A_1}
     \mathrm{ker}\,T_{2}\lra\ldots$$
and 
 $$\mathrm{coker}\,\bfT:\ldots\lra L_{-1}/\mathrm{im}\,T_{-1}\xrightarrow{Q_{-1}}
     L_{0}/\mathrm{im}\,T_{0}\xrightarrow{Q_0}
     L_{1}/\mathrm{im}\,T_{1}\lra\ldots,$$
where, for convenience of notation, we use again $Q_j$ to denote the induced operator on the quotient space. 

We want to use the occasion to correct an erroneous statement  present in the literature, stating that 
the Fredholm property of the mapping cone is equivalent to the Fredholm property of both kernel an cokernel 
of the morphism. In fact, this is not true, in general, as can be seen by this simple example: 
Let $H$ and $L$ be Hilbert spaces and take $\bfT$ as  
$$
\begin{CD}
0 @>>> 0 @>0>> H @>-1>> H @>>> 0 \\
@.  @VV0V  @VV{T_{1}}V @VV0V @. \\ 
0 @>>> L @>1>> L @>0>> 0 @>>> 0, 
\end{CD}
$$
where $1$ denotes the identity maps on $H$ and $L$, respectively. 
The mapping cone associated with this morphism is 
$$
\begin{CD}
0 @>>> 0 @>0>> \begin{matrix}H\\ \oplus\\ L \end{matrix} 
@>\mbox{$\begin{pmatrix}1&0\\T_1&1 \end{pmatrix}$}>> 
\begin{matrix}H\\ \oplus\\ L \end{matrix} @>>> 0 @>>> 0. 
\end{CD}
$$
Obviously, this complex is exact for every choice of $T_1\in\scrL(H,L)$, since 
the block-matrix is always invertible $($we see here also that the Fredholmness, 
respectively exactness, of a mapping cone does not imply the closedness of the 
images $\mathrm{im}\,T_j)$. 
The kernel complex $\mathrm{ker}\,\bfT$ is  
$$
\begin{CD}
0 @>>> 0 @>0>> \mathrm{ker}\,T_1 @>-1>> H @>>> 0. 
\end{CD}
$$
it is exact only if $T_1=0$, it is Fredholm only when $\mathrm{ker}\,T_1$ has finite 
codimension in $H$, i.e., if $\mathrm{im}\,T_1$ is finite-dimensional. 
 If the range of $T_1$ is closed, then $\mathrm{coker}\,\bfT$ is the complex 
$$
\begin{CD}
0 @>>> L @>\pi>> L/\mathrm{im}\,T_1 @>0>> 0 @>>> 0 
\end{CD}
$$
where $\pi$ is the canonical quotient map. Thus $\mathrm{coker}\,\bfT$ is exact only for $T_1=0$; 
it is Fredholm only when $\mathrm{im}\,T_1$ has finite dimension.

Hence, for the equivalence of the Fredholm properties, additional assumptions are required. 
The assumptions employed in the following proposition are optimal, 
as shown again by the above $($counter-$)$example. 

\begin{proposition}\label{prop:mapping_cone}
Assume that, for every $j$, $\mathrm{im}\,T_j$ is closed and that 
\begin{equation}\label{eq:assumption} 
 \mathrm{dim}\,\frac{Q_j^{-1}(\mathrm{im}\,T_{j+1})}{\mathrm{ker}\,Q_j+\mathrm{im}\,T_j}<+\infty. 
\end{equation} 
Then the following properties are equivalent: 
\begin{itemize}
\item[a$)$] The mapping cone $\frakC_\bfT$ associated with $\bfT$ is Fredholm. 
\item[b$)$] Both complexes $\mathrm{ker}\,\bfT$ and $\mathrm{coker}\,\bfT$ are Fredholm. 
\end{itemize}
In case the quotient space in \eqref{eq:assumption} is trivial, the cohomology spaces satisfy 
\begin{align}\label{eq:cohom}
 \scrH^j(\frakC_\bfT) &\cong \scrH^j(\mathrm{ker}\,\bfT)\oplus \scrH^{j-1}(\mathrm{coker}\,\bfT). 
\end{align}
In particular, if the involved complexes are Fredholm and finite,  
 $$\mathrm{ind}\,\frakC_\bfT=\mathrm{ind}\,\mathrm{ker}\,\bfT-\mathrm{ind}\,\mathrm{coker}\,\bfT.$$
Moreover, $\frakC_\bfT$ is exact if, and only if, both $\mathrm{ker}\,\bfT$ and $\mathrm{coker}\,\bfT$
 are exact. 
\end{proposition} 
\begin{proof}
Let us first consider the case where the quotient space in \eqref{eq:assumption} is trivial. 
Then there exist closed subspaces $V_j$ of $L_j$ such that 
$L_j=V_j\oplus\mathrm{im}\,T_j$ and $Q_j:V_j\to V_{j+1}$, for every $j$. 
In fact, choosing a complement $V_j^\prime$ of $\mathrm{im}\,T_j\cap\mathrm{ker}\,Q_j$ in $\mathrm{ker}\,Q_j$ 
for every $j$, take $V_j:=V_j^\prime\oplus Q_j^{-1}(V_{j+1}^\prime)$. It is straightforward to see that 
the complex 
$$
\begin{CD}
\frakQ_V:\quad\ldots @>>> V_{-1} @>Q_{-1}>> V_0 @>Q_{0}>> V_1 @>Q_{1}>> V_2 @>>> \ldots 
\end{CD}
$$
has the same cohomology groups as $\mathrm{coker}\,\bfT$ from above.
Then consider the morphism $\bfS:\mathrm{ker}\,\bfT\to\frakQ_V$ defined by 
$$
\begin{CD}
\ldots @>>> \mathrm{ker}\,T_{j} @>A_{j}>> \mathrm{ker}\,T_{j+1} @>>> \ldots \\
@.  @VV0V  @VV0V  \\ 
\ldots @>>> V_{j} @>Q_{j}>> V_{j+1}   @>>> \ldots 
\end{CD}
$$
$($note that in the vertical arrows we could also write the $T_j$, since they vanish on their kernel$)$. The mapping cone 
$\frakC_\bfS$ is a subcomplex of $\frakC_\bfT$. The quotient complex $\frakC_\bfT/\frakC_\bfS$ is easily seen to be 
the mapping cone of the morphism 
\begin{align}\label{eq:quotient}
\begin{CD}
\ldots @>>> H_j/\mathrm{ker}\,T_{j} @>A_{j}>> H_{j+1}/\mathrm{ker}\,T_{j+1} @>>> \ldots \\
@.  @VVT_jV  @VVT_{j+1}V  \\ 
\ldots @>>> \mathrm{im}\,T_j @>Q_{j}>> \mathrm{im}\,T_{j+1}   @>>> \ldots 
\end{CD}
\end{align}
again by $A_j$ and $T_j$ we denote here the induced maps on the respective quotient spaces. 
Note that all vertical maps are isomorphisms, hence the associated mapping cone is exact. To see this, note that 
$\begin{pmatrix}-A_j&0\\T_j&Q_{j-1}\end{pmatrix}\begin{pmatrix}u\\v\end{pmatrix}=0$
implies that \mbox{$T_ju+Q_{j-1}v=0$}, i.e., $u=-T_{j}^{-1}Q_{j-1}v$. Thus 
\begin{align*}
 \mathrm{ker}\,\begin{pmatrix}-A_j&0\\T_j&Q_{j-1}\end{pmatrix}
 &\subseteq\mathrm{im}\,\begin{pmatrix}-T_j^{-1}Q_{j-1}\\ 1\end{pmatrix}
     =\mathrm{im}\,\begin{pmatrix}-A_{j-1}T_{j-1}^{-1}\\ 1\end{pmatrix}\\
 &=\mathrm{im}\,\begin{pmatrix}-A_{j-1}\\T_{j-1}\end{pmatrix}
     \subseteq\mathrm{im}\,\begin{pmatrix}-A_{j-1}&0\\T_{j-1}&Q_{j-2}\end{pmatrix}
     \subseteq\mathrm{ker}\,\begin{pmatrix}-A_j&0\\T_j&Q_{j-1}\end{pmatrix},
\end{align*}
showing that $\scrH^j(\frakC_\bfT/\frakC_\bfS)=0$. 
Summing up, we have found a short exact sequence of complexes, 
\begin{align}\label{eq:exact}
\begin{CD}
0 @>>> \frakC_\bfS @>\alpha>> \frakC_\bfT @>\beta>> \frakC_\bfT/\frakC_\bfS @>>> 0,  
\end{CD}
\end{align}
where $\alpha$ is the embedding and $\beta$ the quotient map. Since the quotient is an exact complex,  
a standard result of homology-theory $($cf. Corollary 4.5.5 in \cite{Span}, for instance$)$ 
states that the cohomology of $\frakC_\bfS$ and $\frakC_\bfT$ coincide. Since the maps defining 
$\frakC_\bfS$ are just 
 $$\begin{pmatrix}-A_j&0\\0&Q_{j-1}\end{pmatrix}: 
     \begin{matrix}\mathrm{ker}\,T_j\\ \oplus\\ V_{j-1}\end{matrix}
     \lra 
     \begin{matrix}\mathrm{ker}\,T_{j+1} \\ \oplus\\ V_{j}\end{matrix},$$
the claimed relation \eqref{eq:cohom} for the cohomology spaces follows immediately. 
The equivalence of a$)$ and b$)$ is then evident. 

Now let us consider the general case. Choose closed subspaces $U_j$, $V^\prime_j$ and $W_j$ of 
$Q_j^{-1}(\mathrm{im}\,T_{j+1})$ such that
\begin{align*} 
 \mathrm{im}\,T_j&=U_j\oplus(\mathrm{im}\,T_j\cap\mathrm{ker}\,Q_j),\\ 
 \mathrm{ker}\,Q_j&=V_j^\prime\oplus(\mathrm{im}\,T_j\cap\mathrm{ker}\,Q_j),\\
 Q_j^{-1}(\mathrm{im}\,T_{j+1})&=W_j\oplus(\mathrm{im}\,T_j+\mathrm{ker}\,Q_j),
\end{align*}
and define the spaces $V_j:=V_j^\prime\oplus Q_j^{-1}(V_{j+1}^\prime)$. 
Then $L_j=V_j\oplus\mathrm{im}\,T_j\oplus W_j$ and $Q_j:V_j\to V_{j+1}$. As above, consider the complex 
$\frakQ_V$ and the morphism  $\bfS:\mathrm{ker}\,\bfT\to\frakQ_V$; for the cohomology one finds 
$\scrH^j(\mathrm{coker}\,\bfT)=\scrH^{j}(\frakQ_V)\oplus W_j$; note that the $W_j$ are of finite dimension.  
The quotient complex $\frakC_\bfT/\frakC_\bfS$ is the mapping cone of the morphism 
$$
\begin{CD}
\ldots @>>> H_j/\mathrm{ker}\,T_{j} @>A_{j}>> H_{j+1}/\mathrm{ker}\,T_{j+1} @>>> \ldots \\
@.  @VVT_{j}V  @VVT_{j+1}V  \\ 
\ldots @>>> \mathrm{im}\,T_j\oplus W_j @>Q_{j}>> \mathrm{im}\,T_{j+1}\oplus W_{j+1}   @>>> \ldots 
\end{CD}
$$
Since it differs from the exact complex \eqref{eq:quotient} only by the finite-dimensional spaces $W_j$, it is 
a Fredholm complex. By Theorem 4.5.4 of \cite{Span} we now find the exact sequence 
 $$\ldots\lra\scrH^{j-1}(\frakC_\bfT/\frakC_\bfS)\xrightarrow{\;\partial_*\;}
     \scrH^j(\frakC_\bfS)\xrightarrow{\;\alpha_*\;}\scrH^j(\frakC_\bfT)\xrightarrow{\;\beta_*\;}
     \scrH^{j}(\frakC_\bfT/\frakC_\bfS)\xrightarrow{\;\partial_*\;}\ldots$$ 
where $\partial_*$ is the connecting homomorphism for cohomology. Since both spaces 
$\scrH^{j-1}(\frakC_\bfT/\frakC_\bfS)$ and $\scrH^{j}(\frakC_\bfT/\frakC_\bfS)$ are finite-dimensional, 
we find that $\alpha_*$ has finite-dimensional kernel and finite-codimensional range. 
Thus $\scrH^j(\frakC_\bfT)$ is of finite dimension if, and only if, 
$\scrH^j(\frakC_\bfS)$ is. The latter coincides with $\scrH^j(\mathrm{ker}\,\bfT)\oplus\scrH^{j-1}(\frakQ_V)$, which differs 
from $\scrH^j(\mathrm{ker}\,\bfT)\oplus\scrH^{j-1}(\mathrm{coker}\,\bfT)$ only by $W_j$. 
This shows the equivalence of a$)$ and b$)$ in the general case. 
\end{proof}

Of course, condition \eqref{eq:assumption} is void in case all spaces $L_j$ are finite-dimensional. However, 
for the formula of the index established in the proposition, as well as the stated equivalence of exactness, one 
still needs to require that the quotient space in \eqref{eq:assumption} is trivial. 

\begin{remark}\label{rem:isomorphism}
Assume that $\bfT:\frakA\to\frakQ$ is an isomorphism, i.e., all operators $T_j$ are isomorphisms. 
If $\frakP$ is a parametrix to 
$\frakA$, cf.\ Definition $\mathrm{\ref{def:paramterix}}$, then the operators 
 $$S_j:=T_jP_j T_{j+1}^{-1},\qquad j\in\gz,$$
define a parametrix $\frakS$ of the complex $\frakQ$. 
\end{remark}

\subsection{Families of complexes}\label{sec:02.3}

The concept of Hilbert space complexes generalizes to Hilbert \emph{bundle} complexes, i.e. sequences of maps 
\begin{equation*}
 \frakA:\ldots\lra E_{-1}\xrightarrow{A_{-1}}
 E_0\xrightarrow{A_0}E_1\xrightarrow{A_1}E_2
 \xrightarrow{A_2}E_3\lra\ldots,
\end{equation*}
where the $E_j$ are finite or infinite dimensional smooth Hilbert bundles and the $A_j$ are bundle morphisms. 
For our purposes it will be sufficient to deal with the case where all involved bundles have identical base spaces, 
say a smooth manifold $X$, and each $A_j$ preserves the fibre over $x$ for any $x\in X$. 
In this case, by restriction to the fibres, we may associate with $\frakA$ a family of complexes
\begin{equation*}
 \frakA_x:\ldots\lra E_{-1,x}\xrightarrow{A_{-1}}
 E_{0,x}\xrightarrow{A_{0}}E_{1,x}\xrightarrow{A_{1}}E_{2,x}
 \xrightarrow{A_2}E_{3,x}\lra\ldots,\qquad x\in X. 
\end{equation*}
For this reason we shall occasionally call $\frakA$ a family of complexes. It is called a \emph{Fredholm family} 
if $\frakA_x$ is a Fredholm complex for every $x\in X$. Analogously we define an \emph{exact family}. 

Though formally very similar to Hilbert space complexes, families of complexes are  
more difficult to deal with. This is mainly due to the fact that the cohomology spaces $\scrH_j(\frakA_x)$ may 
change quite irregularly with $x$.

\section{Complexes on manifolds with boundary}\label{sec:03}

We shall now turn to the study of complexes of pseudodifferential operators on manifolds with boundary and 
associated boundary value problems.  

\subsection{Boutet de Monvel's algebra with global projection conditions}\label{sec:03.1}

The natural framework for our analysis of complexes on manifolds with boundary is Boutet de Monvel's 
extended algebra with generalized APS-conditions. In the following we provide a compact account on this 
calculus.  

\subsubsection{Boutet de Monvel's algebra}\label{sec:03.1.1}

First we shall present the standard Boutet de Monvel algebra; for details we refer the reader to the existing 
literature, for example \cite{Grub}, \cite{ReSc},  \cite{Schr}.

Let $\Omega$ be a smooth, compact Riemannian manifold with boundary. We shall work with operators 
 \begin{equation}\label{eq:bdm-mapping01}
  \scrA=
  \begin{pmatrix}
   A_++G & K \\
   T & Q
  \end{pmatrix}\;:\quad
  \begin{matrix}
   \scrC^\infty(\Omega,E_0)\\
   \oplus\\
   \scrC^\infty(\pO,F_0)
  \end{matrix}
  \longrightarrow
  \begin{matrix}
   \scrC^\infty(\Omega,E_1)\\
   \oplus\\
   \scrC^\infty(\pO,F_1)
  \end{matrix}, 
 \end{equation}
where $E_j$ and $F_j$ are Hermitean vector bundles over $\Omega$ and $\pO$, respectively, which are  
allowed to be zero dimensional. Every such operator has an order, denoted by $\mu\in\gz$, and a type, 
denoted by 
$d\in\gz$.\footnote{The concept of negative type can be found in \cite{Grub90}, \cite{Grub}, for example.} 
In more detail, 
\begin{itemize}
 \item $A_+$ is the ``restriction" to the interior of $\Omega$ of a $\mu$-th order, classical 
  \emph{pseudodifferential operator} $A$ defined on the smooth double $2\Omega$,  having the 
  two-sided transmission property with respect to $\pO$, 
 \item $G$ is a \emph{Green operator} of order $\mu$ and type $d$, 
 \item $K$ is a $\mu$-th order \emph{potential operator}, 
 \item $T$ is a \emph{trace operator} of order $\mu$ and type $d$, 
 \item $Q\in L^\mu_\cl(\pO;F_0,F_1)$ is a $\mu$-th order, classical \emph{pseudodifferential operator on the boundary}. 
\end{itemize}
The space of all such operators we shall denote by 
 $$\scrB^{\mu,d}(\Omega;(E_0,F_0),(E_1,F_1)).$$ 
The scope of the following example is to illustrate the significance of order and type in this calculus.  

\begin{example}
Let $A=A_+$ be a differential operator on $\Omega$ with coefficients smooth up to the boundary. 
\begin{itemize}
 \item[a$)$] Let $A$ be of order $2$. We shall explain how both Dirichlet and Neumann problem for $A$ are included in 
  Boutet de Monvel's algebra. To this end let 
   $$\gamma_0 u:=u|_{\pO},\qquad \gamma_1 u:=\frac{\partial u}{\partial\nu}\Big|_{\pO}$$
  denote the operators of restriction to the boundary of functions and their derivative in direction of the exterior 
  normal, respectively. Moreover, let $S_j\in L^{3/2-j}_\cl(\pO)$, $j=0,1$, be invertible pseudodifferential 
  operators on the boundary of $\Omega$. Then 
   $$T_j:=S_j\gamma_j \colon\scrC^\infty(\Omega)\lra\scrC^\infty(\pO)$$
  are trace operators of order $2$ and type $j+1$. If $E_0=E_1:=\cz$, $F_1:=\cz$ and $F_0:=\{0\}$, then  
  $\calA_j:=\begin{pmatrix}A\\ T_j\end{pmatrix}$ belongs to $\scrB^{2,j+1}(\Omega;(\cz,0),(\cz,\cz))$. 
  In case $\calA_j$ is invertible, the inverses are of the form 
   $$\calA_j^{-1}=\begin{pmatrix}P_++G_j & K_j \end{pmatrix}
       \,\in\,\scrB^{-2,0}(\Omega;(\cz,\cz),(\cz,0));$$
 for the original Dirichlet and Neumann problem one finds 
  $$\begin{pmatrix}A\\ \gamma_j\end{pmatrix}^{-1}=\begin{pmatrix}P_++G_j & K_jS_j \end{pmatrix}.$$
 \item[b$)$] Let $A$ now have order $4$ and consider $A$ jointly with Dirichlet and Neumann condition. 
  We define 
   $$T:=\begin{pmatrix}S_0\gamma_0\\ S_1\gamma_1\end{pmatrix}\colon\scrC^\infty(\Omega)\lra
       \begin{matrix}\scrC^\infty(\pO)\\\oplus\\ \scrC^\infty(\pO)\end{matrix}
       \cong \scrC^\infty(\pO,\cz^2)$$
  with pseudodifferential isomorphisms $S_j\in L^{7/2-j}_\cl(\pO)$. Then $T$ is a trace operator of order $4$ 
  and type $2$, and  $\begin{pmatrix}A\\ T\end{pmatrix}$ belongs to $\scrB^{4,2}(\Omega;(\cz,0),(\cz,\cz^2))$. 
  The discussion of invertibility is similar as in a$)$. 
\end{itemize}
At first glance, the use of the isomorphisms $S_j$ may appear strange but, indeed, is just a choice of 
normalization of orders; it could be replaced by any other choice of normalization, resulting in a straightforward 
reformulation. 
\end{example}

As a matter of fact, with $\scrA\in\scrB^{\mu,d}(\Omega;(E_0,F_0),(E_1,F_1))$ as in 
\eqref{eq:bdm-mapping01} is associated a \emph{principal symbol} 
\begin{equation}\label{eq:bdm-symbol01}
 \sigma^\mu(\scrA)=\big(\sigma^\mu_\psi(\scrA),\sigma^\mu_\partial(\scrA)\big),
\end{equation}
that determines the ellipticity of $\calA$ $($see below$)$; the components are 
\begin{enumerate}
 \item the usual \emph{homogeneous principal symbol} of the pseudodifferential operator 
  $A$ $($restricted to $S^*\Omega$, the unit co-sphere bundle of $\Omega)$,  
    $$\sigma^\mu_\psi(\scrA):=\sigma^\mu_\psi(A):\pi_\Omega^*E_0\lra\pi_\Omega^*E_1,$$
  where $\pi_\Omega:S^*\Omega\to\Omega$ is the canonical projection, 
 \item the so-called \emph{principal boundary symbol} which is a vector bundle morphism 
   \begin{equation}\label{eq:bdm-boundarysymb01}
        \sigma^\mu_\partial(\scrA):
        \begin{matrix}\pi_{\pO}^*(\scrS(\rz_+)\otimes E_0^\prime)\\ \oplus\\ \pi_{\pO}^* F_0\end{matrix}
        \lra
        \begin{matrix}\pi_{\pO}^*(\scrS(\rz_+)\otimes E_1^\prime)\\ \oplus\\ \pi_{\pO}^* F_1\end{matrix},
   \end{equation}
    where $\pi_{\pO}:S^*\pO\to\pO$ again denotes the canonical projection and $E_j^\prime=E_j|_{\pO}$ 
    is the restriction of $E_j$ to the boundary. 
\end{enumerate}

\subsubsection{Boutet de Monvel's algebra with APS conditions}\label{sec:03.1.2}

This extension of Boutet de Monvel's algebra has been introduced in \cite{Schu37}. 
Consider two pseudodifferential projections $P_j\in L^0_\cl(\pO;F_j,F_j)$, $j=0,1$,  on the boundary of $\Omega$. 
We denote by 
 $$\scrB^{\mu,d}(\Omega;(E_0,F_0;P_0),(E_1,F_1;P_1))$$
the space of all operators $\scrA\in\scrB^{\mu,d}(\Omega;(E_0,F_0),(E_1,F_1))$ such that 
 $$\scrA(1-\scrP_0)=(1-\scrP_1)\scrA=0,\qquad \scrP_j:=\begin{pmatrix}1&0\\0&P_j\end{pmatrix}.$$
If we denote by
 $$\scrC^\infty(\pO,F_j;P_j):=P_j\big(\scrC^\infty(\pO,F_j)\big)$$
the range spaces of the projections $P_j$, which are closed subspaces, then any such $\scrA$ induces 
continuous maps
 \begin{equation}\label{eq:bdm-mapping02}
  \scrA
  \colon
  \begin{matrix}
   \scrC^\infty(\Omega,E_0)\\
   \oplus\\
   \scrC^\infty(\pO,F_0;P_0)
  \end{matrix}
  \longrightarrow
  \begin{matrix}
   \scrC^\infty(\Omega,E_1)\\
   \oplus\\
   \scrC^\infty(\pO,F_1;P_1)
  \end{matrix}.  
 \end{equation}
For sake of clarity let us point out that $\scrA$ acts also as an operator as in \eqref{eq:bdm-mapping01} but 
it is the mapping property \eqref{eq:bdm-mapping02} in the subspaces determined by the projections which 
is the relevant one. 

The use of the terminology ``algebra" originates from the fact that operators can be composed in 
the following sense:
\begin{theorem}
Composition of operators induces maps 
\begin{align*}
 \scrB^{\mu_1,d_1}&(\Omega;(E_1,F_1;P_1),(E_2,F_2;P_2))\times 
 \scrB^{\mu_0,d_0}(\Omega;(E_0,F_0;P_0),(E_1,F_1;P_1))\\
 &\lra \scrB^{\mu_0+\mu_1,d}(\Omega;(E_0,F_0;P_0),(E_2,F_2;P_2)),
\end{align*}
where the resulting is $d=\max(d_0,d_1+\mu_0)$.  
\end{theorem}

The Riemannian and Hermitian metrics allow us to define $L_2$-spaces $($and then $L_2$-Sobolev spaces$)$ 
of sections of the bundles over $\Omega$. Identifying these spaces with their dual spaces, as usually done for 
Hilbert spaces, we can associate with $\scrA$ its formally adjoint operator $\scrA^*$. Then the following is true:  

\begin{theorem}
Let $\mu\le0$. Taking the formal adjoint induces maps 
\begin{align*}
 \scrB^{\mu,0}&(\Omega;(E_0,F_0;P_0),(E_1,F_1;P_1))\lra 
 \scrB^{\mu,0}(\Omega;(E_1,F_1;P_1^*),(E_0,F_0;P_0^*)),  
\end{align*} 
where $P_j^*$ is the formal adjoint of the projection $P_j$.  
\end{theorem}

Let us now describe the principal symbolic structure of the extended algebra. Since the involved $P_j$ 
are projections, also their associated principal symbols $\sigma^0_\psi(P_j)$ are projections 
$($as bundle morphisms$)$; thus their ranges define subbundles 
\begin{equation}\label{eq:FP}
 F_j(P_j):=\sigma^0_\psi(P_j)\big(\pi_{\pO}^*F_j\big)\subseteq \pi_{\pO}^*F_j.
\end{equation}
Note that, in general, $F_j(P_j)$ is not a pull-back to the co-sphere bundle of a bundle over the boundary $\pO$.   

The principal  boundary symbol of $\scrA\in\scrB^{\mu,d}(\Omega;(E_0,F_0;P_0),(E_1,F_1;P_1))$, which initially is 
defined as in \eqref{eq:bdm-boundarysymb01}, restricts then to a morphism 
\begin{equation}\label{eq:bdm-boundarysymb02}
 \begin{matrix}\pi_{\pO}^*(\scrS(\rz_+)\otimes E_0)\\ \oplus\\ F_0(P_0)\end{matrix}
  \lra
 \begin{matrix}\pi_{\pO}^*(H^{s-\mu}(\rz_+)\otimes E_1)\\ \oplus\\ F_1(P_1)\end{matrix}. 
\end{equation}
This restriction we shall denote by $\sigma_\partial^\mu(\scrA;P_0,P_1)$ and will call it again the 
\emph{principal boundary symbol} of $\scrA$; the \emph{principal symbol} of $\scrA$ is then the tuple 
\begin{equation}\label{eq:bdm-symbol02}
 \sigma^\mu(\scrA;P_0,P_1)=\big(\sigma^\mu_\psi(\scrA),\sigma^\mu_\partial(\scrA;P_0,P_1)\big). 
\end{equation}
The two components of the principal symbol behave multiplicatively under composition and are 
compatible with the operation of taking the formal adjoints in the obvious way. 

\begin{definition}\label{def:bdm-elliptic}
$\scrA\in\scrB^{\mu,d}(\Omega;(E_0,F_0;P_0),(E_1,F_1;P_1))$ is called \emph{$\sigma_\psi$-elliptic} 
if $\sigma^\mu_\psi(\scrA)$ is an isomorphism. It is called 
\emph{elliptic} if additionally $\sigma^\mu_\partial(\scrA;P_0,P_1)$ is an isomorphism. 
\end{definition}

\subsubsection{Sobolev spaces and the fundamental theorem of elliptic theory}\label{sec:03.1.3}

In the following we let $H^s(\Omega,E)$ and $H^s(\pO,F)$ with $s\in\gz$ denote the standard scales of 
$L_2$-Sobolev spaces of sections in the bundles $E$ and $F$, respectively. Moreover, 
$H^s_0(\Omega,E)$ denotes the closure of $\scrC^\infty_0(\mathrm{int}\,\Omega,E)$ in $H^s(\Omega,E)$. 

Let $\scrA\in\scrB^{\mu,d}(\Omega;(E_0,F_0;P_0),(E_1,F_1;P_1))$. The range spaces 
 $$H^s(\pO,F_j;P_j):=P_j\big(H^s(\pO,F_j)\big)$$
are closed subspaces of $H^s(\pO,F_j)$, and $\scrA$ induces continuous maps 
\begin{equation}\label{eq:bdm-cont01}
 \begin{matrix}
       H^s(\Omega,E_0)\\ \oplus\\ H^{s}(\partial\Omega,F_0;P_0)
     \end{matrix}
     \lra
     \begin{matrix}
       H^{s-\mu}(\Omega,E_0)\\ \oplus\\ H^{s-\mu}(\partial\Omega,F_1;P_1)
     \end{matrix},
     \qquad s\ge d. 
\end{equation}
Similarly, the principal boundary symbol $\sigma^\mu(\scrA;P_0,P_1)$ induces morphisms 
\begin{equation}\label{eq:bdm-cont02}
 \begin{matrix}\pi_{\pO}^*(H^s(\rz_+)\otimes E_0)\\ \oplus\\ F_0(P_0)\end{matrix}
  \lra
 \begin{matrix}\pi_{\pO}^*(H^{s-\mu}(\rz_+)\otimes E_1)\\ \oplus\\ F_1(P_1)\end{matrix},
 \qquad s\ge d. 
\end{equation}
As a matter of fact, in the above Definition \ref{def:bdm-elliptic} of ellipticity it is equivalent considering the 
principal boundary symbol as a map \eqref{eq:bdm-boundarysymb01} or as a map \eqref{eq:bdm-cont02} 
for some fixed integer $s\ge d$. 

\begin{theorem}\label{thm:bdm-main}
For $\scrA\in\scrB^{\mu,d}(\Omega;(E_0,F_0;P_0),(E_1,F_1;P_1))$ the following statements are equivalent$:$
\begin{itemize}
 \item[a$)$] $\scrA$ is elliptic. 
 \item[b$)$] There exists an $s\ge\max(\mu,d)$ such that the map \eqref{eq:bdm-cont01} associated with $\scrA$ is 
  Fredholm. 
 \item[c$)$] For every $s\ge\max(\mu,d)$ the map \eqref{eq:bdm-cont01} associated with $\scrA$ is Fredholm. 
 \item[d$)$] There is an $\scrB\in \scrB^{-\mu,d-\mu}(\Omega;(E_1,F_1;P_1),(E_0,F_0;P_0))$  
  such that 
  \begin{align*}
   \scrB\scrA-\scrP_0 &\in \scrB^{-\infty,d}(\Omega;(E_0,F_0;P_0),(E_0,F_0;P_0)),\\
   \scrA\scrB-\scrP_1 &\in \scrB^{-\infty,d-\mu}(\Omega;(E_1,F_1;P_1),(E_1,F_1;P_1)).
  \end{align*}
\end{itemize}
Any such operator $\scrB$ is called a parametrix of $\scrA$.  
\end{theorem}

\begin{remark}
By $($formally$)$ setting $E_0$ and $E_1$ equal to zero, the above block-matrices reduce to the entry in 
the lower-right corner. The calculus thus reduces to one for pseudodifferential operators on the boundary. 
We shall use the notation $L^\mu_\cl(\pO;F_0,F_1)$ and $L^\mu_\cl(\pO;(F_0;P_0),(F_1;P_1))$, 
respectively. The ellipticity of $Q\in L^\mu_\cl(\pO;(F_0;P_0),(F_1;P_1))$ is then desribed by one symbol only, 
namely $\sigma_\psi^\mu(Q):F_0(P_0)\to F_1(P_1)$, cf.\ \eqref{eq:FP}. 
\end{remark}

\subsection{Example: The Cauchy-Riemann operator on the unit disc}\label{sec:03.1.5}

Let us discuss a simple example. Let $\Omega$ be the unit-disc in $\rz^2$ and 
$A=\overline\partial=(\partial_x+i\partial_y)/2$ be the Cauchy-Riemann operator. Identify the Sobolev spaces 
$H^s(\pO)$ with the cooresponding spaces of Fourier series, i.e., 
 $$f\in H^s(\pO) \iff \big(|n|^{s}\wh{f}(n)\big)_{n\in\gz}\in \ell^2(\gz).$$
The so-called \emph{Calder\'{o}n projector} $C$, defined by 
 $$\wh{Cf}(n)=\begin{cases}\wh{f}(n)&\quad: n\ge 0\\ 0&\quad: n<0\end{cases}$$
belongs to $L^0_\cl(\pO)$ and satisfies $C=C^2=C^*$. 
\forget{
Since $\gamma_0$ induces an isomorphism between the kernel of $A$ acting on $H^s(\Omega)$, $s\ge 1$, and 
$H^{s-1/2}(\pO;C)$ it is obvious that 
$\begin{pmatrix}A\\ \gamma_0\end{pmatrix}\colon
     H^s(\Omega)\lra
    \begin{matrix}H^{s-1}(\Omega)\\ \oplus\\ H^{s-1/2}(\pO;C)\end{matrix}$ 
is an isomorphism. To unify orders, let $S\in L^{1/2}_\cl(\pO)$ be invertible, $P:=SCS^{-1}$, and 
$T_0:=S\gamma_0$. 
}
Note that $\gamma_0$ induces an isomorphism between the kernel of $A$ acting on $H^s(\Omega)$, $s\ge 1$, and 
$H^{s-1/2}(\pO;C)$. To unify orders, let $S\in L^{1/2}_\cl(\pO)$ be invertible, $P:=SCS^{-1}$, and 
$T_0:=S\gamma_0$. Then 
\begin{equation}\label{eq:C-R-1}
 \begin{pmatrix}A\\ T_0\end{pmatrix}\colon
     H^s(\Omega)\lra
    \begin{matrix}H^{s-1}(\Omega)\\ \oplus\\ H^{s-1}(\pO;P)\end{matrix}
\end{equation}
is an isomorphism for any integer $s\ge 1$.   

\begin{lemma}
Let $T=P\phi\gamma_0$ with $\phi:H^{s-1/2}(\pO)\to H^{s-1}(\pO)$ beinga bounded operator. Then the map from 
\eqref{eq:C-R-1} with $T_0$ replaced by $T$ is Fredholm if, and only if, 
$P\phi:H^{s-1/2}(\pO;C)\to H^{s-1}(\pO;P)$ is Fredholm. 
\end{lemma}
\begin{proof}
Let $\begin{pmatrix}B&K\end{pmatrix}$ be the inverse of \eqref{eq:C-R-1}. Then the Fredholmness of 
$\begin{pmatrix}A\\ T\end{pmatrix}$ is equivalent to that of 
$\begin{pmatrix}A\\ T\end{pmatrix}\begin{pmatrix}B&K\end{pmatrix}=\begin{pmatrix}1&0\\  TB& TK\end{pmatrix}$ 
in $H^{s-1}(\Omega)\oplus H^{s-1}(\pO;P)$, i.e., to that of 
$TK:H^{s-1}(\pO;P)\to H^{s-1}(\pO;P)$. But now $1=T_0K=S\gamma_0K$ on $H^{s-1}(\pO;P)$ implies that 
$TK=P\phi S^{-1}$ on $H^{s-1}(\pO;P)$. It remains to observe that $S^{-1}:H^{s-1}(\pO;P)\to H^{s-1}(\pO;C)$ 
isomorphically. 
\end{proof}

Let us now interpret the previous observation within the framework of the Boutet de Monvel algebra with generalized APS 
conditions. Let $P\in L^0_\cl(\pO)$ be an arbitrary projection with $P-C\in L^{-1}_\cl(\pO)$, i.e., $P$ has homogeneous principal 
symbol
 $$\sigma_\psi^0(P)(\theta,\tau)=\sigma_\psi^0(C)(\theta,\tau)
     =\begin{cases}1&\quad: \tau=1\\ 0&\quad: \tau=-1\end{cases},$$
where we use polar coordinates on $\pO$ and $\tau$ denotes the covariable to $\theta$. 
By a straightforward calculation we find that the boundary smbol of $A$ is 
 $$\sigma^1_\partial(A)(\theta,\tau)=-\frac{1}{2}e^{i\theta}(\partial_t+\tau)\colon 
     \scrS(\rz_{+,t})\lra\scrS(\rz_{+,t}),\qquad \tau=\pm 1,$$
and therefore is surjective with kernel 
 $$\mathrm{ker}\,\sigma_\partial^1(A)(\theta,\tau)
     =\begin{cases}\mathrm{span}\,e^{-t}&\quad: \tau=1\\ 0&\quad: \tau=-1 \end{cases}.$$

\begin{lemma}
Let $T=B\gamma_0$ with $B\in L^{1/2}_\cl(\pO)$ and 
 $$\scrA:=\begin{pmatrix}A\\P T\end{pmatrix} \;\in\;\scrB^{1,1}(\Omega;(\cz,0;1),(\cz,\cz;P)).$$
The following properties are equivalent$:$
\begin{itemize}
 \item[a$)$] $\scrA$ is elliptic. 
 \item[b$)$] $\sigma^{1/2}_\psi(B)(\theta,1)\not=0$. 
 \item[c$)$] $PBC\in L^{1/2}_\cl(\pO;(\cz;C),(\cz;P))$ is elliptic. 
 \item[d$)$] $PB:H^{s-1/2}(\pO;C)\to H^{s-1}(\pO;P)$ is Fredholm for all $s$. 
\end{itemize}
\end{lemma}
\begin{proof}
Clearly, the homogeneous principal symbol of $A$ never vanishes. The principal boundary symbol is 
given by 
\begin{align*}
\sigma^1_\partial(\scrA)(\theta,-1)
 &=\begin{pmatrix}\sigma^1_\partial(A)(\theta,-1)\\0\end{pmatrix}\colon
      \scrS(\rz_+)\lra\begin{matrix}\scrS(\rz_+)\\ \oplus \\ \{0\}\end{matrix},\\
\sigma^1_\partial(\scrA)(\theta,1)
 &=\begin{pmatrix}\sigma^1_\partial(A)(\theta,1)\\ \sigma_\psi^{1/2}(B)(\theta,1)\gamma_0\end{pmatrix}\colon
      \scrS(\rz_+)\lra\begin{matrix}\scrS(\rz_+)\\ \oplus \\ \cz\end{matrix}, 
\end{align*}
where $\gamma_0u=u(0)$ for every $u\in\scrS(\rz_+)$. Thus ellipticity of $\scrA$ is equivalent to the non-vanishing 
of $\sigma^{1/2}_\psi(B)(\theta,1)$. The remaining equivalences are then clear. 
\end{proof}

\subsection{Boundary value problems for complexes}\label{sec:03.2}

In the following we shall consider a complex 
\begin{equation}\label{eq:bvp-complexA}
     \frakA:0\lra H^s(\Omega,E_0)\xrightarrow{A_{0}}
     H^{s-\nu_0}(\Omega,E_1)\xrightarrow{A_1}
     \ldots\xrightarrow{A_n}
     H^{s-\nu_{n}}(\Omega,E_{n+1})\lra0
\end{equation}
with $A_j=\wt{A}_{j,+}+G_j\in B^{\mu_j,d_j}(\Omega;E_j,E_{j+1})$  
and $\nu_j:=\mu_0+\ldots+\mu_j$, where $s$ is assumed to be so large that all mappings have sense 
(i.e., $s\ge\nu_j$ and $s\ge d_{j}+\nu_{j-1}$ for every $j=0,\ldots,n)$.  

\begin{definition}
The complex $\frakA$ is called $\sigma_\psi$-elliptic, if the associated family of complexes made up by the 
homogeneous principal symbols $\sigma_\psi^{\mu_j}(A_j)$, which we shall denote by $\sigma_\psi(\frakA)$, 
is an exact family.  
\end{definition}

Let us now state one of the main theorems of this section, concerning the existence and structure of complementing 
boundary conditions. 

\begin{theorem}\label{thm:complex-main}
Let $\frakA$ as in \eqref{eq:bvp-complexA} be $\sigma_\psi$-elliptic. 
\begin{itemize}
\item[a$)$] There exist bundles $F_1,\ldots,F_{n+2}$ and projections $P_j\in L^0_\cl(\pO;F_j,F_j)$ 
 such that the complex $\frakA$ can be completed to a Fredholm morphism 
 $($in the sense of Section $\mathrm{\ref{sec:02.2}})$
$$
\begin{CD}
0 @>>>  H_0 @>A_{0}>> H_1 @>A_{1}>> \ldots @>A_{n}>> H_{n+1} @>>> 0 \\
@. @VV{T_{0}}V @VV{T_{1}}V @. @VV{T_{n+1}}V \\ 
0 @>>>  L_0 @>Q_{0}>> L_1 @>Q_{1}>> \ldots @>Q_{n}>> L_{n+1} @>>> 0 
\end{CD}
$$
where we use the notation 
 $$H_j:=H^{s-\nu_{j-1}}(\Omega,E_{j}),\qquad L_j:=H^{s-\nu_{j}}(\pO,F_{j+1};P_{j+1}),$$
the $T_j$ are trace operators of order $\mu_j$ and type $0$ and 
 $$Q_j\in L_\cl^{\mu_{j+1}}(\pO;(F_{j+1};P_{j+1}),(F_{j+2};P_{j+2})).$$ 
In fact, all but one of the $P_j$ can be chosen to be the identity. 
Moreover, it is possible to choose all projections equal to the identity if, and only if, the 
index bundle of $\frakA$ satisfies 
 $$ \mathrm{ind}_{S^*\partial\Omega}\,\sigma_\partial(\mathfrak{A})
     \in\pi^* K(\partial\Omega),$$
where $\pi:S^*\partial\Omega\to\Omega$ is the canonical projection.
\item[b$)$] A statement analogous to a$)$ holds true with the trace operators $T_j$ substituted 
by $K_j$ with potential operators $K_j:L_j\to H_j$ of order $-\mu_j$ 
\end{itemize}
\end{theorem}

The main part of the proof will be given in the next Sections \ref{sec:03.2.1} and \ref{sec:03.2.2}. Before, 
let us first explain why, in fact, it suffices to demonstrate part b$)$ of the previous theorem in case all 
orders $\mu_j$, types $d_j$, and the regularity $s$ are equal to zero. Roughly speaking, 
this is possible by using order-reductions and by passing to adjoint complexes. 
In detail, the argument is as follows:

We shall make use of a certain family of isomorphism, whose existence is proved, for example, in 
Theorem 2.5.2 of \cite{Grub}:  there are operators 
$\Lambda^m_j\in B^{m,0}(\Omega;E_j,E_{j})$, $m\in\gz$, which are invertible in the algebra with  
$(\Lambda^m_j)^{-1}=\Lambda^{-m}_j$ and which induce isomorphisms 
$H^s(\Omega,E_j)\to H^{s-m}(\Omega,E_j)$ for every $s\in\rz$. Their adjoints, denoted by 
$\Lambda^{m,*}_j$, are then isomorphisms 
$\Lambda^{m,*}_j:H^{m-s}_0(\Omega,E_j)\to H^{-s}_0(\Omega,E_j)$ for every $s\in\rz$ and also 
$\Lambda^{m,*}_j\in B^{m,0}(\Omega;E_j,E_j)$. 

Assume now that Theorem \ref{thm:complex-main}.b$)$ holds true in case $\mu_j=d_j=s=0$. 

Given the complex $\frakA$ from \eqref{eq:bvp-complexA}, consider the new complex 
\begin{equation*}
     \wt\frakA:0\lra L^2(\Omega,E_0)\xrightarrow{\wt{A}_{0}}
     L^2(\Omega,E_1)\xrightarrow{\wt{A}_1}
     \ldots\xrightarrow{\wt{A}_n}
     L^2(\Omega,E_{n+1})\lra0,
\end{equation*}
where $\wt{A}_j:=\Lambda^{s-\nu_j}_{j+1} A_j \Lambda^{\nu_{j-1}-s}_j$. The $\wt{A}_j$ have order and 
type $0$ and $\wt{\frakA}$ is $\sigma_\psi$-elliptic. Thus there are projections $\wt{P}_j$ and block-matrices 
 $$\wt{\calA}_j=\begin{pmatrix}-\wt{A}_j&\wt{K}_{j+1}\\0&\wt{Q}_{j+1}\end{pmatrix}
     \;\in\;
     \scrB^{0,0}(\Omega;(E_j,F_{j+2};\wt{P}_{j+2}),(E_{j+1},F_{j+3};\wt{P}_{j+3}))$$
$($with $j=-1,\ldots,n)$ that form a Fredholm complex. Now choose families of invertible pseudodifferential 
operators  
$\lambda_j^{r}\in L^r_\cl(\pO;F_{j+1},F_{j+1})$, $r\in\rz$, with $(\lambda_j^{r})^{-1}=\lambda_j^{-r}$. 
Then also the 
 $${\calA}_j:=
     \begin{pmatrix}\Lambda^{\nu_j-s}_{j+1}&0\\0&\lambda^{\nu_{j+2}-s}_{j+2}\end{pmatrix}
     \begin{pmatrix}-\wt{A}_j&\wt{K}_{j+1}\\0&\wt{Q}_{j+1}\end{pmatrix}
     \begin{pmatrix}\Lambda^{s-\nu_{j-1}}_{j}&0\\0&\lambda^{s-\nu_{j+1}}_{j+1}\end{pmatrix}$$
form a Fredholm complex. This shows b$)$ in the general case with the choice of  
\begin{equation}\label{eq:potential-operator}
 K_j:=\Lambda_j^{\nu_{j-1}-s}\wt{K}_j\lambda_j^{s-\nu_j},
\end{equation}
the projections $P_{j}=\lambda_{j-1}^{\nu_{j-1}-s}\wt{P}_{j}\lambda_{j-1}^{s-\nu_{j-1}}$ and 
 $$Q_j:=\lambda_{j+1}^{\nu_{j+1}-s}\wt{Q}_j\lambda_j^{s-\nu_j}\,\in\,
     L^{\mu_{j+1}}_\cl(\pO;(F_{j+1},P_{j+1}),(F_{j+2},P_{j+2})).$$ 

Now let us turn to a$)$. In case $\mu_j=d_j=s=0$ pass to the adjoint complex
\begin{equation*}
     0\lra L^2(\Omega,\wt{E}_{0})\xrightarrow{{B}_{0}}
     L^2(\Omega,\wt{E}_1)\xrightarrow{{B}_1}
     \ldots\xrightarrow{{B}_n}
     L^2(\Omega,\wt{E}_{n+1})\lra0,
\end{equation*}
with $\wt{E}_j=E_{n+1-j}$ and $B_j=A_{n-j}^*$. Apply to this complex part b$)$ of the Theorem, with 
bundles $\wt{F}_j=F_{n+3-j}$ and projections $\wt{P}_j$ for $1\le j\le n+2$, resulting 
in a complex of block-matrices $\calB_j=\begin{pmatrix}-{B}_j&{K}_{j+1}\\0&\wt{Q}_{j+1}\end{pmatrix}$. 
Then also the $\calA_{j}:=\wt{\calB}_{n-j}^*$, $j=0,\ldots,n+1$, form a Fredholm complex and a$)$ follows 
with $T_j:=K_{n+1-j}^*$, $Q_j:=\wt{Q}^*_{n+1-j}$ and projections $P_j:=\wt{P}_{n+3-j}^*$. 

Finally, consider the general case of a$)$. First define 
$\wt{A}_j=\Lambda^{s-\nu_j}_{j+1} A_j \Lambda^{\nu_{j-1}-s}_j$ 
as above and then pass to the adjoint complex of the $\wt{B}_j:=\wt{A}_{n-j}^*$. 
Using b$)$, this leads to a Fredholm complex of operators 
 $$\wt{\calB}_j=\begin{pmatrix}-\wt{B}_j&\wt{K}_{j+1}\\0&\wt{Q}_{j+1}\end{pmatrix}
     \;\in\;
     \scrB^{0,0}(\Omega;(\wt{E}_j,\wt{F}_{j+2};\wt{P}_{j+2}),(\wt{E}_{j+1},\wt{F}_{j+3};\wt{P}_{j+3})).$$
Now we define 
 $${\calB}_j=
     \begin{pmatrix}
     \wt{\Lambda}_{j+1}^{s-\nu_{n-j-1},*}&0\\0&\wt{\lambda}^{s-\nu_{n-j-1},*}_{j+3}
     \end{pmatrix}
     \begin{pmatrix}-\wt{B}_j&\wt{K}_{j+1}\\0&\wt{Q}_{j+1}\end{pmatrix}
     \begin{pmatrix}\wt{\Lambda}_{j}^{\nu_{n-j}-s,*}&0\\0&\wt{\lambda}^{\nu_{n-j}-s,*}_{j+2}\end{pmatrix},$$
where the operators $\wt{\Lambda}^m_j$ refer to the bundle $\wt{E}_j$, while $\wt{\lambda}^r_j$ to the bundle 
$\wt{F}_{j}$. These $\calB_j$ then define a Fredholm complex acting as operators  
 $$\calB_j\colon
     \begin{matrix}H^{\nu_{n-j}-s}_0(\Omega,\wt{E}_j)\\ \oplus \\ H^{\nu_{n-j}-s}(\pO,\wt{F}_{j+2},{P}_{j+2}^\prime)\end{matrix}
     \lra
     \begin{matrix}H^{\nu_{n-j-1}-s}_0(\Omega,\wt{E}_{j+1})\\ \oplus \\ H^{\nu_{n-j-1}-s}(\pO,\wt{F}_{j+3},{P}_{j+2}^\prime)\end{matrix}$$
with resulting projections $P_j^\prime$. Now observe that 
$\calB_j^*=\begin{pmatrix}-A_{n-j}& 0\\  K_{j+1}^*& Q_{j+1}^*\end{pmatrix}$ with 
\begin{align*}
 Q_{j+1}=\wt{\lambda}^{s-\nu_{n-j-1},*}_{j+3} \wt{Q}_{j+1}\wt{\lambda}^{\nu_{n-j}-s,*}_{j+2},\qquad
 K_{j+1}=\wt{\Lambda}^{s-\nu_{n-j-1},*}_{j+1} \wt{K}_{j+1}\wt{\lambda}^{\nu_{n-j}-s,*}_{j+2},
\end{align*}
and that $\wt{\Lambda}^{s-\nu_{n-j-1},*}_{j+1} \wt{K}_{j+1}\wt{\lambda}^{\nu_{n-j-1}-s,*}_{j+2}$ is a potential operator of order $0$, mapping
 $$H^{\nu_{n-j-1}-s}(\pO,\wt{F}_{j+2},{P}_{j+2}^\prime)\lra 
     H^{\nu_{n-j-1}-s}_0(\Omega,\wt{E}_{j+1}).$$ 
We conclude that  
 $$T_j:=K_{n+1-j}^*: H^{s-\nu_{j-1}}(\Omega,E_j)\lra H^{s-\nu_{j}}(\pO,{F}_{j+1},{P}_{j+1}),
     \quad P_j:=(P_{n-j+3}^\prime)^*,$$
is a trace operator of order $\mu_j$ and type $0$ and the result follows by redefining $Q_{n-j+1}^*$ as $Q_j$.     

\begin{remark}
The use of order reductions in the above discussion leads to the fact that the operators $T_j$ and $Q_j$ 
constructed in Theorem $\mathrm{\ref{thm:complex-main}}$.a$)$ depend on the regularity $s$. 
However, once constructed them for some fixed choice $s=s_0$, it is a consequence of the general theory 
presented below in Section $\mathrm{\ref{sec:06.4.1}}$, that the resulting boundary value problem induces a 
Fredholm morphism not only for the choice $s=s_0$ but for all admissible $s$. 
An analogous comment applies to part b$)$ of 
Theorem $\mathrm{\ref{thm:complex-main}}$. 
\end{remark}

\subsubsection{The index element of a $\sigma_\psi$-elliptic complex}\label{sec:03.2.1}
 
We start out with the $\sigma_\psi$-elliptic complex 
\begin{equation*}
     \frakA:0\lra L^2(\Omega,E_0)\xrightarrow{A_{0}}
     L^2(\Omega,E_1)\xrightarrow{A_1}
     \ldots\xrightarrow{A_n}
     L^2(\Omega,E_{n+1})\lra0, 
\end{equation*}
with $A_j\in B^{0,0}(\Omega;E_j,E_{j+1})$. The associated principal boundary symbols $\sigma_\partial^0(A_j)$ 
form the family of complexes  
\begin{equation*}
     \sigma_\partial(\frakA):0\lra\calE_0\xrightarrow{\sigma_\partial^0(A_{0})}
     \calE_1\xrightarrow{\sigma_\partial^0(A_1)}
     \ldots\xrightarrow{\sigma_\partial^0(A_n)}
     \calE_{n+1}\lra0, 
\end{equation*}
where we have used the abbreviation
 $$\calE_j:=\pi^*\big(L^2(\rz_+,E_j|_{\pO}\big),\qquad \pi\colon S^*\pO\lra\pO.$$
Due to the $\sigma_\psi$-ellipticity, $\sigma_\partial(\frakA)$ is a Fredholm family. 

\begin{theorem}\label{thm:index-element}
There exist non-negative integers $\ell_1,\ldots,\ell_{n+1}$ and principal boundary symbols  
 $$\fraka_j=\begin{pmatrix}-\sigma_\partial^0(A_j)&k_{j+1}\\0&q_{j+1}\end{pmatrix},\qquad j=0,\ldots,n,$$
of order and type $0$ such that 
 $$0\lra 
     \begin{matrix}\calE_0\\ \oplus\\ \cz^{\ell_1}\end{matrix}
     \xrightarrow{\fraka_0}
     \begin{matrix}\calE_1\\ \oplus\\ \cz^{\ell_2}\end{matrix}
     \xrightarrow{\fraka_1}\ldots 
     \xrightarrow{\fraka_{n-1}} 
     \begin{matrix}\calE_n\\ \oplus\\ \cz^{\ell_{n+1}}\end{matrix}
     \xrightarrow{\fraka_n}
     \begin{matrix}\calE_{n+1}\\ \oplus\\ 0\end{matrix}
     \lra 0$$
is a family of complexes which is exact in every position but possibly the first, with finite-dimensional kernel bundle 
$J_0:=\mathrm{ker}\,\fraka_0$. 
In particular, the index-element of $\mathfrak{A}$ is given by
\begin{equation}\label{eq:def-index}  
 \mathrm{ind}_{S^*\partial\Omega}\,\sigma_\partial(\mathfrak{A})
     =[J_0]+\sum_{j=1}^{n+1}(-1)^j[\cz^{\ell_j}].
\end{equation}
\end{theorem}
\begin{proof} 
For notational convenience let us write $a_j:=-\sigma_\partial^0(A_j)$. 
The proof is an iterative procedure that complements, one after the other, the principal boundary 
symbols $a_n,a_{n-1},\ldots a_0$ to block-matrices. 

Since $\sigma_\partial(\frakA)$ is a Fredholm family,  
$a_n:\calE_n\lra\calE_{n+1}$
has fibrewise closed range of finite co-dimension. It is then a well-known fact, cf.\ Subsection 3.1.1.2 of \cite{ReSc} for example, 
that one can choose a principal potential symbol $k_{n+1}:\cz^{\ell_{n+1}}\to\calE_{n+1}$ such that 
 $$\begin{pmatrix}a_n&k_{n+1}\end{pmatrix}:
     \begin{matrix}\calE_n\\\oplus\\ \cz^{\ell_{n+1}}\end{matrix}\lra\calE_{n+1}$$
is surjective. Choosing $q_n:=0$ this defines $\fraka_n$. For $n=0$ this finishes the proof. So let us assume 
$n\ge1$. 

Set $\ell_{n+2}:=0$. 
Let us write $\wt{\calE}_j:=\calE_j\oplus\cz^{\ell_{j+1}}$ and assume that, 
for an integer $1\le i\le n$, we have constructed $\fraka_i,\ldots,\fraka_n$ such that 
 $$\wt{\calE}_{i}\xrightarrow{\fraka_i}\wt{\calE}_{i+1}\xrightarrow{\fraka_{i+1}}\ldots 
     \xrightarrow{\fraka_n} \wt{\calE}_{n+1}\lra 0$$
is an exact family. Then the families of Laplacians 
 $$\frakd_j=\fraka_{j-1}\fraka_{j-1}^*+\fraka_{j}^*\fraka_{j},\qquad i+1\le j\le n+1,$$
are fibrewise  isomorphisms, i.e., bijective principal boundary symbols. 
Thus also the inverses $\frakd_j^{-1}$ are principal boundary symbols. 
Then the principal boundary symbols 
 $$\pi_j:=1-\fraka_j^*\frakd_{j+1}^{-1}\fraka_j,\qquad i\le j\le n,$$
are fibrewise the orthogonal projections in $\wt{\calE}_j$ onto the kernel of 
$\fraka_j$; we shall verify this in detail at the end of the proof. 

Now consider the morphism  
 $$\begin{matrix}\calE_{i-1}\\ \oplus\\ \wt{\calE}_{i+1}\end{matrix}
     \xrightarrow{\begin{pmatrix}a_{i-1}&\fraka_{i}^*\end{pmatrix}}
     \wt{\calE_{i}},$$
where $\calE_i$ is considered as a subspace of  $\wt{\calE_{i}}=\calE_i\oplus\cz^{\ell_{i+1}}$. 
Since $a_{i-1}$ maps into the kernel of $\fraka_i$, while fibrewise the image of $\fraka_i^*$ is the orthogonal 
complement of the kernel of $\fraka_i$, we find that 
 $$\mathrm{im}\,\begin{pmatrix}a_{i-1}&\fraka_{i}^*\end{pmatrix}
     =\mathrm{im}\,a_{i-1}\oplus (\mathrm{ker}\,\fraka_{i})^\perp$$
is fibrewise of finite codimension in 
$\wt{\calE_{i}}=\mathrm{ker}\,\fraka_{i}\oplus (\mathrm{ker}\,\fraka_{i})^\perp$. 
Therefore, there exists an integer $\ell_{i}$ and a principal boundary symbol 
 $$b=\begin{pmatrix}k\\ q\end{pmatrix}:\cz^{\ell_{i}}
     \lra\wt\calE_{i}=\begin{matrix}\calE_i\\ \oplus\\\cz^{\ell_{i+1}}\end{matrix}$$
$($in particular, $k$ is a principal potential symbol$)$ such that 
 $$\begin{matrix}\calE_{i-1}\\ \oplus\\ \wt{\calE}_{i+1}\\ \oplus \\ \cz^{\ell_{i}}\end{matrix}
     \xrightarrow{\begin{pmatrix}a_{i-1}&\fraka_{i}^* & b \end{pmatrix}} \wt{\calE_{i}}$$
is surjective. We now define  
 $$\begin{pmatrix}k_{i}\\ q_{i}\end{pmatrix}:=\pi_i b: 
     \cz^{\ell_{i}}
     \lra\wt\calE_{i}=\begin{matrix}\calE_i\\ \oplus\\\cz^{\ell_{i+1}}\end{matrix}$$
and claim that 
 $$\fraka_{i-1}:=\begin{pmatrix}a_{i-1}&k_{i}\\0&q_{i}\end{pmatrix}:
     \begin{matrix}\wt\calE_{i-1}\\ \oplus\\ \cz^{\ell_{i}}\end{matrix}\lra 
     \mathrm{ker}\,\fraka_{i}$$
surjectively. In fact, by construction, $\fraka_{i-1}$ maps into the kernel of $\fraka_i$. 
Moreover, given $x$ in a fibre of $\mathrm{ker}\,\fraka_i$, there exists $(u,v,w)$ in the corresponding fibre 
of  $\calE_{i-1} \oplus \wt{\calE}_{i+1} \oplus \cz^{\ell_{i}}$ such that 
 $$x=a_{i-1}u+\fraka_i^*v+bw.$$
Being $\pi_i$ the orthogonal projection on the kernel of $\fraka_i$, we find 
 $$x=\pi_ix=a_{i-1}u+\pi_ibw=\fraka_{i-1}\begin{pmatrix}u\\w\end{pmatrix}.$$
Thus we have constructed $\fraka_{i-1}$ such that 
 $$\wt{\calE}_{i-1}\xrightarrow{\fraka_{i-1}}\wt{\calE}_{i}
     \xrightarrow{\fraka_i}\wt{\calE}_{i+1}\xrightarrow{\fraka_{i+1}}\ldots 
     \xrightarrow{\fraka_n} \wt{\calE}_{n+1}\lra 0$$
is an exact complex. Now repeat this procedure until $a_0$ has been modified. 

It remains to check that the $\pi_j$ in fact are projections as claimed:  
Clearly $\pi_j=1$ on $\mathrm{ker}\,\fraka_j$. Moreover, 
\begin{align*}
 \fraka_j^*\fraka_j\pi_j
 &=\fraka_j^*\fraka_j-\fraka_j^*\fraka_j\fraka_j^*\frakd_{j+1}^{-1}\fraka_j\\
 &=\fraka_j^*\fraka_j
     -\fraka_j^*(\underbrace{\fraka_j\fraka_j^*+\fraka_{j+1}^*\fraka_{j+1}}_{=\frakd_{j+1}})
       \frakd_{j+1}^{-1}\fraka_j
     +\underbrace{\fraka_j^*\fraka_{j+1}^*}_{=(\fraka_{j+1}\fraka_{j})^*=0}
        \fraka_{j+1}\frakd_{j+1}^{-1}\fraka_j=0. 
\end{align*}
Hence $\pi_j$ maps into $\mathrm{ker}\,\fraka_j^*\fraka_j=\mathrm{ker}\,\fraka_j$. Finally 
\begin{align*}
 (1-\pi_j)^2
 &=\fraka_j^*\frakd_{j+1}^{-1}\fraka_j\fraka_j^*\frakd_{j+1}^{-1}\fraka_j\\
 &=\fraka_j^*\frakd_{j+1}^{-1}(\underbrace{\fraka_j\fraka_j^*+\fraka_{j+1}^*\fraka_{j+1}}_{=\frakd_{j+1}})
      \frakd_{j+1}^{-1}\fraka_j-\fraka_j^*\frakd_{j+1}^{-1}\fraka_{j+1}^*\fraka_{j+1}\frakd_{j+1}^{-1}\fraka_j
   =1-\pi_j, 
\end{align*}
since $\fraka_j$ maps into $\mathrm{ker}\,\fraka_{j+1}$ and 
$\frakd_{j+1}^{-1}:\mathrm{ker}\,\fraka_{j+1}\to\mathrm{ker}\,\fraka_{j+1}$, hence 
$\fraka_{j+1}\frakd_{j+1}^{-1}\fraka_j=0$. The proof of Theorem \ref{thm:index-element} is complete. 
\end{proof}

\subsubsection{The proof of Theorem $\mathrm{\ref{thm:complex-main}}$}\label{sec:03.2.2}

Let us now turn to the proof of Theorem \ref{thm:complex-main}.b$)$ in the case of $\mu_j=d_j=s=0$. 
In fact, the statement it is a consequence of the following Theorem \ref{thm:complex-main-order-zero} which 
is slightly more precise. In its proof we shall apply some results for complexes on manifolds with boundary 
which we shall provide in Section \ref{sec:06.4} below; these results in turn are a consequence of our general 
theory for complexes in operator algebras developed in Sections \ref{sec:05} and \ref{sec:06}. 

\begin{theorem}\label{thm:complex-main-order-zero}
Let notation be as in Section $\mathrm{\ref{sec:03.2.1}}$. Then there exist non-negative integers 
$\ell_0,\ldots,\ell_{n+1}$, operators 
 $$\calA_j=\begin{pmatrix}-A_j&K_{j+1}\\ 0&Q_{j+1}\end{pmatrix}
     \;\in\; \calB^{0,0}(\Omega;(E_j,\cz^{\ell_{j+1}}),(E_{j+1},\cz^{\ell_{j+2}})),\qquad j=0,\ldots,n,$$
and 
 $$\calA_{-1}=\begin{pmatrix}0&K_{0}\\0&Q_{0}\end{pmatrix}
     \;\in\; \calB^{0,0}(\Omega;(0,\cz^{\ell_{0}},P_0),(E_{0},\cz^{\ell_{1}}))$$
with a projection $P_0\in L^0(\partial\Omega;\cz^{\ell_{0}},\cz^{\ell_{0}})$ such that 
 $$0\lra\begin{matrix}0\\ \oplus\\ L^2(\pO,\cz^{\ell_0};P_0)\end{matrix}
     \xrightarrow{\calA_{-1}}\ldots 
     \xrightarrow{\calA_{j}}
     \begin{matrix}L^2(\Omega,E_{j+1})\\ \oplus\\ L^2(\pO,\cz^{\ell_{j+2}})\end{matrix}
     \ldots\xrightarrow{\calA_n}
     \begin{matrix}L^2(\Omega,E_{n+1})\\ \oplus\\ 0\end{matrix}
     \lra 0$$
is a Fredholm complex. If, and only if,  
\begin{equation}\label{eq:index-element}
 \mathrm{ind}_{S^*\partial\Omega}\,\sigma_\partial(\frakA)\in\pi^* K(\partial\Omega),
\end{equation}
i.e., the index-element of the complex $\frakA$ belongs to the pull-back of the $K$-group of the boundary under 
the canonical projection $\pi:S^*\partial\Omega\to\partial\Omega$, we may replace $\cz^{\ell_0}$ by a vector 
bundle $F_0$ over $\partial\Omega$ and $P_0$ by the identity map. 
\end{theorem}
\begin{proof}
Repeating the construction in the proof of Theorem \ref{thm:index-element}, we can find $\ell_0$ and 
a boundary symbol 
 $$\fraka_{-1}=\begin{pmatrix}k_0\\q_0\end{pmatrix}:\cz^{\ell_0}\lra 
     \begin{matrix}\calE_0\\ \oplus\\ \cz^{\ell_1}\end{matrix},\qquad 
     \mathrm{im}\,\begin{pmatrix}k_0\\q_0\end{pmatrix}=\mathrm{ker}\,\fraka_0=J_0.$$
Therefore, 
 $$\cz^{\ell_0}=\mathrm{ker}\,\begin{pmatrix}k_0\\q_0\end{pmatrix}\oplus 
     \Big(\mathrm{ker}\,\begin{pmatrix}k_0\\q_0\end{pmatrix}\Big)^\perp\cong 
     \mathrm{ker}\,\begin{pmatrix}k_0\\q_0\end{pmatrix}\oplus J_0.$$
Now let $P_0$ be a projection whose principal symbol concides with the projection 
onto $J_0$ $($such a projection exists, cf.\ the appendix in \cite{Schu37}, for instance$)$. 
Then Proposition \ref{prop:lift_bdm_rev} 
implies the existence of $\calA_j$ as stated, forming a complex which is both $\sigma_\psi$- and 
$\sigma_\partial$-elliptic. Then the complex is Fredholm due to Theorem \ref{thm:complex_bdm_rev}. 

In case \eqref{eq:index-element} is satisfied, there exists an integer $L$ such that $J_0\oplus\cz^L$ is a pull-back 
of a bundle $F_0$ over $\partial\Omega$, i.e. $J_0\oplus\cz^L\cong\pi^*G$. Now replace $\ell_0$ and $\ell_1$ by 
$\ell_0+L$ and $\ell_1+L$, respectively. Extend $k_0$ and $k_1,\, q_1$ by 0 from $\cz^{\ell_0}$ to 
$\cz^{\ell_0}\oplus\cz^L$ and $\cz^{\ell_1}$ to $\cz^{\ell_1}\oplus\cz^L$, respectively. Moreover, extend $q_0$ 
to $\cz^{\ell_0}\oplus\cz^L$ by letting $q_0=1$ on $\cz^L$. After these modifications, rename $\ell_j+L$ by 
$\ell_j$ for $j=0,1$ as well as the extended $k_0$ and $q_0$ by $\wt{k}_0$ and $\wt{q}_0$, respectively. 
We obtain that 
 $$W:=\Big(\mathrm{ker}\,\begin{pmatrix}\wt{k}_0\\ \wt{q}_0\end{pmatrix}\Big)^\perp
     \cong J_0\oplus\cz^L\cong\pi^*F_0.$$
With an isomorphism $\alpha:\pi^*F_0\to W$ we then define the boundary symbol 
 $$\fraka_{-1}=\begin{pmatrix}{k}_0\\ {q}_0\end{pmatrix}
     :=\begin{pmatrix}\wt{k}_0\\ \wt{q}_0\end{pmatrix}\Big|_W\circ\alpha$$
and again argue as above to pass to a Fredholm complex of operators $\calA_j$. 
\end{proof}

\subsubsection{General boundary value problems}\label{sec:03.2.3}

We have seen in Theorem \ref{thm:complex-main} that any $\sigma_\psi$-elliptic complex \eqref{eq:bvp-complexA} 
can be completed to a boundary value problem which results to be Fredholm. Vice versa, given a boundary value 
problem for $\frakA$, we can characterize when it is Fredholm: 

\begin{theorem}\label{thm:complex-main_2}
Let $\frakA$ as in \eqref{eq:bvp-complexA}.  
\begin{itemize}
\item[a$)$] Assume we are given a boundary value problem 
   $$
      \begin{CD}
      0 @>>>  H_0 @>A_{0}>> H_1 @>A_{1}>> \ldots @>A_{n}>> H_{n+1} @>>> 0 \\
     @. @VV{T_{0}}V @VV{T_{1}}V @. @VV{T_{n+1}}V \\ 
      0 @>>>  L_0 @>Q_{0}>> L_1 @>Q_{1}>> \ldots @>Q_{n}>> L_{n+1} @>>> 0 
     \end{CD}
  $$
  with spaces 
   $$H_j:=H^{s-\nu_{j-1}}(\Omega,E_{j}),\qquad L_j:=H^{s-\nu_{j}}(\pO,F_{j+1};P_{j+1}),$$
  trace operators of order $\mu_j$ and type $d_j$, and 
   $$Q_j\in L_\cl^{\mu_{j+1}}(\pO;(F_{j+1};P_{j+1}),(F_{j+2};P_{j+2})).$$ 
  Then the following statements are equivalent$:$
   \begin{itemize}
    \item[$1)$] The boundary value problem is Fredholm 
    \item[$2)$] $\frakA$ is $\sigma_\psi$-elliptic and the family of complexes generated by the boundary 
     symbols 
       $$\begin{pmatrix}
           -\sigma_\partial^{\mu_j}(A_j)&\sigma_\partial^{\mu_j}(T_{j+1};P_{j+2})\\
           0&\sigma_\partial^{\mu_{j+1}}(Q_{j+1};P_{j+2},P_{j+3})
          \end{pmatrix},\qquad j=-1,\ldots,n,$$      
      associated with the mapping cone is an exact family.
   \end{itemize}
\item[b$)$] A statement analogous to a$)$ holds true with the trace operators $T_j:H_j\to L_j$ substituted 
by potential operators $K_j:L_j\to H_j$ of order $-\mu_j$. 
\end{itemize}
\end{theorem}

In fact, this theorem is a particular case of Theorem \ref{thm:complex_bdm_rev} $($applied to the associated 
mapping cone$)$. 

\subsection{Example: The deRham complex}\label{sec:03.3}

Let $\mathrm{dim}\,\Omega=n+1$ and $E_k$ denote the $k$-fold exterior product of the $($complexified$)$ 
co-tangent bundle; sections in $E_k$ are complex differential forms of degree $k$ over $\Omega$. 
Let $d_k$ denote the operator of external differentiation on $k$-forms. The deRham complex 
\begin{equation*}
     0\lra H^s(\Omega,E_0)\xrightarrow{d_0}
     H^{s-1}(\Omega,E_1)\xrightarrow{d_1}
     \ldots\xrightarrow{d_n}
     H^{s-n-1}(\Omega,E_{n+1})\lra0
\end{equation*}
$(s\ge n+1)$ is $\sigma_\psi$-elliptic and the associated principal boundary symbols induce an exact family of 
complexes. Therefore the deRham complex 
is a Fredholm complex without adding any 
additional boundary conditions. However, one can also pose ``Dirichlet-conditions'', i.e., consider  
   $$
      \begin{CD}
      \cdots  @>d_{j-1}>> H^{s-j}(\Omega,E_j) @>d_j>> H^{s-j-1}(\Omega,E_{j+1}) @>d_{j+1}>> \cdots \\
      @. @VV{R_{j}}V  @VV{R_{j+1}}V  @.\\ 
      \cdots  @>d_{j-1}>> H^{s-j-1/2}(\pO,F_j) @>d_j>> H^{s-j-3/2}(\pO,F_{j+1}) @>d_{j+1}>> \cdots 
     \end{CD}
  $$
where the second row is the deRham-complex on the boundary and the $R$'s map forms on $\Omega$  
to their tangential 
part. This is also a Fredholm problem whose index coincides with the Euler characteristic of the pair $(\Omega,\pO)$.
Note that for meeting the set-up of Theorem \ref{thm:complex-main_2} one needs to replace the $R_k$ by 
$T_k:=S_kR_k$ and the differentials $d_{k-1}$ on the boundary by 
$Q_k:=S_kd_{k-1}S_{k-1}^{-1}$ with some invertible pseudodifferential operators 
$S_k\in L^{1/2}_\cl(\pO;F_k,F_k)$. 

We omit any details, since all these observations have already been mentioned in Example 9 of \cite{Dyni}. 

\subsection{Example: The Dolbeault complex}\label{sec:03.3.5}

In this section we show that the Dolbeault complex generally violates the Atiyah-Bott obstruction. 

Complex differential forms of bi-degree $(0,k)$ over $\cz^n\cong\rz^{2n}$ are sections in the corresponding 
vector bundle denoted by $E_k$. Let $\dbar_k$ be the dbar-operator acting on $(0,k)$-forms and let 
 $$\Pi\xi=\sum_i \xi\Big(\frac{\partial}{\partial \overline{z}_i}\Big) d\overline{z}_i$$ 
denote the canonical projection in the $($complexified$)$ co-tangent bundle. 
The homogeneous principal symbol of $\dbar_k$ is given by
 $$\sigma_\psi^1(\dbar_k)(\xi)\omega=(\Pi\xi)\wedge\omega,\qquad 
     \xi\in T^*_z\rz^{2n},\quad\omega\in E_{k,z},$$
with $z\in\cz^n$. Now let $\Omega\subset\cz^n$ be a compact domain with smooth boundary and restrict $\dbar_k$ to 
$\Omega$. If $r:\Omega\to\rz$ is a boundary defining function for 
$\Omega$, the principal boundary symbol of $\dbar_k$ is $($up to scaling$)$ given by   
 $$\sigma_\partial^1(\dbar_k)(\xi^\prime)\eta=\Pi\xi^\prime\wedge\eta
     -i(\Pi dr)\wedge\frac{d\eta}{dr},$$
where $\xi^\prime\in T^*_{z^\prime}\partial \Omega$ and $\eta\in H^s(\rz_+)\otimes E_{k,z^\prime}$ 
with $z^\prime\in\partial \Omega$. 

For simplicity let us now take $n=2$ and let $\Omega=\{z\in\cz^2\mid |z|\le 1\}$ be the unit-ball in $\cz^2$. 
Using the generators $\dzbar_1$ and $\dzbar_2$, we shall identify 
$E^{0,0}$ and $E^{0,2}$ with $\cz^2\times\cz$ and $E^{0,1}$ with $\cz^2\times\cz^2$. 
As boundary defining function we take an $r$ with $r(z)=2(|z|-1)$ near $\partial\Omega$; then, on $\pO$,  
 $$\frac{\partial}{\partial r}=\frac{1}{2}\sum_{j=1,2} z_j\frac{\partial}{\partial z_j}
     +\overline{z}_j\frac{\partial}{\partial \overline{z}_j}.$$
Now we identify $T^*\partial\Omega$ with those co-vectors from $T^*\Omega|_{\pO}$ vanishing on $\partial/\partial r$. Hence, 
representing $\xi\in T^*_z \Omega$ as $\xi=\sum_{j=1,2} \xi_j\dzbar_j+\overline{\xi}_jdz_j$ we find 
 $$\xi\in T^*_z\partial \Omega\iff \re(\xi_1\overline{z}_1+\xi_2\overline{z}_2)=0.$$ 
In other words, we may identify $T^*\partial \Omega$ with 
 $$T^*\partial \Omega=\big\{(z,\xi)\in\cz^2\oplus\cz^2\mid |z|=1,\;\re \xi\cdot z=0\big\},$$
where $\xi\cdot z=(\xi,z)_{\cz^2}$ denotes the standard inner product of $\cz^2$; for the unit co-sphere bundle of 
$\partial \Omega$ we additionally require $|\xi|=1$. Note that for convenience we shall use notation $z$ and $\xi$ 
rather than $z^\prime$ and $\xi^\prime$ as above. 

Using all these identifications, the principal boundary symbols $\sigma_\partial^1(\dbar_0)$ and 
$\sigma_\partial^1(\dbar_1)$ can be identified with the operator-families  
 \begin{align*}
  d_0:H^{s+1}(\rz_+)\lra H^{s}(\rz_+,\cz^2),\qquad
  d_1:H^{s}(\rz_+,\cz^2)\lra H^{s-1}(\rz_+),
 \end{align*}
defined on $T^*\partial\Omega$ by
 \begin{align*}
  d_0(z,\xi)u&= (\xi_1u-iz_1u^\prime,\xi_2u-iz_2u^\prime)=\xi u-izu^\prime,\\  
  d_1(z,\xi)v&= \xi_2v_1-\xi_1v_2-i(z_2v^\prime_1-z_1v^\prime_2)=v\cdot\xi^\perp-iv^\prime\cdot z^\perp;  
 \end{align*}
here, $u^\prime=\frac{du}{dr}$ and similarly $v^\prime$ and $v_j^\prime$ denote derivatives with respect to the variable 
$r\in\rz_+$. Moreover, $c^\perp:=(\overline{c}_2,-\overline{c}_1)$ provided that $c=(c_1,c_2)\in\cz^2$. Note that 
$c^\perp\cdot d=-d^\perp\cdot c$ for every $c,d\in\cz^2$; in particular, $c\cdot c^\perp=0$. 

Therefore, the principal boundary symbol of the dbar-complex 
 $$\overline{\frakD}:\quad 0\lra H^{s+1}(\Omega,E^{0,0})
     \xrightarrow{\dbar_0}H^{s}(\Omega,E^{0,1})
     \xrightarrow{\dbar_1}H^{s-1}(\Omega,E^{0,2})\lra 0$$
corresponds to the family of complexes 
\begin{equation}\label{eq:compl}
 \sigma_\partial(\overline{\frakD}):\quad
 0\lra H^{s+1}(\rz_+)\xrightarrow{d_0} H^{s}(\rz_+,\cz^2)\xrightarrow{d_1}H^{s-1}(\rz_+)\lra0. 
\end{equation}
It is easily seen that $\overline{\frakD}$ is $\sigma_\psi$-elliptic. Hence the boundary symbols form a Fredholm family. 
We shall now determine explicitely the index element of $\overline{\frakD}$ and shall see that $\overline{\frakD}$ 
violates the Atiyah-Bott obstruction.   

\begin{proposition}
The complex \eqref{eq:compl} is exact for all $(z,\xi)\in S^*\partial \Omega$ with $z\not=i\xi$, while 
 $$\mathrm{dim}\,\mathrm{ker}\,d_0(i\xi,\xi)
     =\mathrm{dim}\,\frac{\mathrm{ker}\,d_1(i\xi,\xi)}{\mathrm{im}\,d_0(i\xi,\xi)}=1,\qquad 
     \mathrm{im}\,d_1(i\xi,\xi)=H^{s-1}(\rz_+).$$ 
In particular, $d_1$ is surjective in every point of $S^*\partial\Omega$. 
\end{proposition}
\begin{proof}
We will first study range and kernel of $d_0$: 
By definition, 
 $$\mathrm{im}\,d_0(z,\xi)=\Big\{\xi u-izu^\prime\mid u\in H^{s+1}(\rz_+)\Big\}.$$
Clearly $u$ belongs to the kernel of $d_0(z,\xi)$ if, and only if, 
 $$\begin{pmatrix}\xi_1 & -iz_1\\ \xi_2 & -iz_2\end{pmatrix}\begin{pmatrix}u\\ u^\prime\end{pmatrix}
     =\begin{pmatrix}0\\0\end{pmatrix}.$$
In case $\xi$ and $z$ are $($complex$)$ linearly independent, this simply means $u=0$. 

Otherwise there exists a constant $c\in\cz$ with $|c|=1$ such that $z=c\xi$. Then $0=\re z\cdot\xi=\re{c}$ 
shows that $c=\pm i$. In case $z=-i\xi$ we obtain 
 $$\xi_1(u-u^\prime)=\xi_2(u-u^\prime)=0.$$
Since $\xi\not=0$ it follows that $u$ is a multiple of $e^r$. Hence $u=0$ is the only solution in $H^s(\rz_+)$. 
Analogously, in case $z=i\xi$ we find that $u$ must be a multiple of $e^{-r}$, which is always an element of 
$H^s(\rz_+)$. In conclusion, for $(z,\xi)\in S^*\pO$, 
 $$\mathrm{ker}\,d_0(z,\xi)=
     \begin{cases}\mathrm{span}\{e^{-r}\}&\quad: z=i\xi,\\ \{0\}&\quad:  \mathrm{else}\end{cases}.$$
 
Let us next determine range and kernel of $d_1$: It will be useful to use the operators 
 $$L_-w=w-w^\prime,\qquad L_+w=w+w^\prime.$$
Note that $L_-:H^s(\rz_+)\to H^{s-1}(\rz_+)$ is an isomorphism $($recall that 
$L_\pm=\mathrm{op}_+(l_\pm)$ with symbol 
$l_\pm(\tau)=1\pm i\tau$ being so-called plus- and minus-symbols, respectively, 
that play an important role in Boutet de Monvel's calculus$)$. 
 
Let us consider the equation 
\begin{equation}\label{eq:1}
 d_1(z,\xi)v = \xi_2v_1-\xi_1v_2-i(z_2v^\prime_1-z_1v^\prime_2)=f.
\end{equation}
We consider three cases$:$
\begin{itemize}
\item[$($i$)$] Assume that $\xi$ and $z$ are linearly independent, hence $\delta:=i(z_1\xi_2-z_2\xi_1)\not=0$. 
 Let $f\in H^{s-1}(\rz_+)$ be given. If 
  $${v}:=(iz-\xi)L_-^{-1}f/\delta\in H^{s}(\rz_+,\cz^2),$$
 a direct computation shows that 
  $$d_1(z,\xi){v}=L_-^{-1}f-(L_-^{-1}f)^\prime=L_-L_-^{-1}f=f.$$
 Hence $d_1(z,\xi)$ is surjective. 

 Now let $f=0$ and set $w=i(z_2v_1-z_1v_2)$. Then, due to \eqref{eq:1},  
 $w^\prime=\xi_1v_2-\xi_2v_1$. In particular, $w,w^\prime\in H^s(\rz_+)$, i.e., $w\in H^{s+1}(\rz_+)$.  
 Moreover, 
  $$\begin{pmatrix}w \\ w^\prime\end{pmatrix}
      =\begin{pmatrix}-iz_2&iz_1\\ -\xi_2&\xi_1\end{pmatrix}\begin{pmatrix}v_1\\ v_2\end{pmatrix},$$
 which is eqivalent to 
  $$v=(v_1,v_2)=(\xi_1w-iz_1w^\prime,\xi_2w-iz_2w^\prime)/\delta=(\xi w-izw^\prime)/\delta,$$
 hence $\mathrm{ker}\,d_1(z,\xi)=\mathrm{im}\,d_0(z,\xi)$. 
\item[$($ii$)$] Consider the case $z=-i\xi$. Then, setting $w=\xi_2v_1-\xi_1v_2$, \eqref{eq:1} becomes 
 $L_-w=f$. Then, using that $|\xi|=1$, 
 \eqref{eq:1} is equivalent to 
  $$\xi_2(v_1-\overline{\xi}_2L^{-1}_-f)-\xi_1(v_2+\overline{\xi}_1L^{-1}_-f)=0.$$
 Since the orthogonal complement of the span of $\xi^\perp=(\overline{\xi}_2,-\overline{\xi}_1)$ is just the span of 
 $\xi$, we find that the solutions of \eqref{eq:1} are precisely those $v$ with 
  $$v=(v_1,v_2)=\Big(\xi_1\lambda+\overline{\xi}_2L^{-1}_-f,\xi_2\lambda-\overline{\xi}_1L^{-1}_-f\Big),\qquad 
      \lambda\in H^s(\rz_+). $$
 Since $\xi\lambda=\xi L_-u=\xi(u-u^\prime)=\xi u-izu^\prime$ for $\xi=iz$ with $u=L^{-1}_-\lambda$, we 
 conclude that $d_1(z,\xi)$ is surjective with $\mathrm{ker}\,d_1(z,\xi)=\mathrm{im}\,d_0(z,\xi)$.  
\item[$($iii$)$] It remains to consider the case $z=i\xi$. Similarly as before, setting $w=\xi_2v_1-\xi_1v_2$, 
 \eqref{eq:1} becomes $L_+w=w+w^\prime=f$. Note that the general solution is 
  $$w=c e^{-r}+w_f,\qquad w_f(r)=e^{-r}\int_0^r e^sf(s)\,ds=\int_0^r e^{-t}f(r-t)\,dt,$$
 where $f\mapsto w_f:H^{s-1}(\rz_+)\to H^{s}(\rz_+)$ is a continuous right-inverse of $L_+$. Then \eqref{eq:1} 
 is equivalent to 
  $$\xi_2\big(v_1-\overline{\xi}_2(ce^{-r}+w_f)\big)-\xi_1\big(v_2+\overline{\xi}_1(ce^{-r}+w_f)\big)=0$$
 we find that the solutions of \eqref{eq:1} are precisely those $v$ with 
  $$v=(v_1,v_2)=\Big(\xi_1\lambda+\overline{\xi}_2(ce^{-r}+w_f),
      \xi_2\lambda-\overline{\xi}_1(ce^{-r}+w_f)\Big),
      \qquad \lambda\in H^s(\rz_+). $$
 Since $L_+:H^{s+1}(\rz_+)\to H^s(\rz_+)$ surjectively, we can represent any $\xi\lambda$ as 
 $\xi L_+u=\xi (u+u^\prime)=\xi u-iz u^\prime$ and thus conclude that $d_1(z,\xi)$ is surjective with 
 \begin{align*}
  \mathrm{ker}\,d_1(z,\xi)&=\mathrm{im}\,d_0(z,\xi)\oplus\mathrm{span}\{\xi^\perp e^{-r}\}.  
 \end{align*}
\end{itemize}
This finishes the proof of the proposition. 
\end{proof}

In the previous proposition, including its proof, 
we have seen that $d_1(z,\xi)$ is surjective for every $(z,\xi)\in S^*\pO$ with  
 $$\mathrm{ker}\,d_1(z,\xi)=
     \begin{cases}
       \mathrm{im}\,d_0(z,\xi)\oplus\mathrm{span}\{\xi^\perp e^{-r}\}&\quad: z=i\xi,\\ 
       \mathrm{im}\,d_0(z,\xi)&\quad:  \mathrm{else}
     \end{cases}.$$
Now let $\varphi\in\scrC^\infty(\rz)$ be a cut-off function with $\varphi\equiv1$ near $t=0$ and $\varphi(t)=0$ 
if $|t|\ge1/2$. Then define $\phi\in\scrC^\infty(S^*\pO)$ by 
 $$\phi(z,\xi)=\varphi(|\xi+iz|^2/2)=\varphi(1+iz\cdot\xi);$$
for the latter identity recall that $z\cdot\xi=\im\,z\cdot\xi$ for $(z,\xi)\in S^*\pO$. Obviously, $\phi$ is 
supported near the skew-diagonal $\{(i\xi,\xi)\mid |\xi|=1\}\subset S^*\pO$. 

\begin{lemma}
With the above notation define 
 $$v(z,\xi)=\phi(z,\xi)\xi^\perp e^{-ir/z\cdot\xi}\;\in\;\scrS(\rz_+,\cz^2),\qquad (z,\xi)\in S^*\pO$$
 $($recall that $r$ denotes the variable of $\rz_+)$. 
Then we have 
 $$\mathrm{ker}\,d_1(z,\xi)=\mathrm{im}\,d_0(z,\xi)+\mathrm{span}\{v(z,\xi)\}\qquad 
     \forall\;(z,\xi)\in S^*\pO.$$
\end{lemma}
\begin{proof}
Obviously, $v(i\xi,\xi)=\xi^\perp e^{-r}$. Moreover, for $z\not=i\xi$, 
 $$d_1(z,\xi)v(z,\xi)=v(z,\xi)\cdot\xi^\perp-i\frac{dv(z,\xi)}{dr}\cdot z^\perp=0,$$
using $c^\perp\cdot d^\perp=d\cdot c$. 
Hence $v(z,\xi)\in\mathrm{ker}\,d_1(z,\xi)=\mathrm{im}\,d_0(z,\xi)$. 
\end{proof}

If we now define 
 $$k_0\in\scrC^\infty\big(S^*\pO,\scrL(\cz,H^s(\rz_+,\cz^2))\big),\qquad 
     c\mapsto k_0(z,\xi)c:=cv(z,\xi)$$
then 
\begin{equation}\label{eq:compl2}
 0\lra 
 \begin{matrix}H^{s+1}(\rz_+)\\ \oplus\\ \cz\end{matrix}
 \xrightarrow{\wt{d}_0:=(d_0\;\;k_0)} H^{s}(\rz_+,\cz^2)\xrightarrow{\;d_1\;}H^{s-1}(\rz_+)\lra0
\end{equation}
is a family of complexes, which is exact in the second and third position. The index-element of  
$\overline{\frakD}$ is generated by the kernel-bundle of $\wt{d}_0$. 

\begin{lemma}\label{lem:basis}
The kernel-bundle of $\wt{d}_0$ is one-dimensional with 
 $$\mathrm{ker}\,\wt{d}_0(z,\xi)=
     \begin{cases}
       \mathrm{span}\{(e^{-r},0)\}&\quad: z=i\xi,\\ 
       \mathrm{span}\Big\{\Big(\phi(z,\xi)\frac{(z,\xi)}{(z,\xi^\perp)}e^{-ir/(z,\xi)},-1\Big)\Big\}&\quad:  \mathrm{else}
     \end{cases}.$$
\end{lemma}
\begin{proof}
In case $z=i\xi$, the ranges of $k_0$ and $a_0$ have trivial intersection; hence 
 $$\mathrm{ker}\,\wt{d}_0(i\xi,\xi)=\mathrm{ker}\,{d}_0(i\xi,\xi)\oplus\mathrm{ker}\,k_0(i\xi,\xi)= 
     \mathrm{span}\,\{e^{-r}\}\oplus\{0\}.$$
In case $z\not=\pm i\xi$, we find that $d_0(z,\xi)$ has the left-inverse 
 $$d_0(z,\xi)^{-1}v=\frac{1}{(\xi,z^\perp)}(v,z^\perp),$$
since if $v=d_0(z,\xi)u=\xi u-izu^\prime$ then $d_0(z,\xi)^{-1}v=u$ by simple computation. 
Thus 
 $$\wt{d}_0(z,\xi)\begin{pmatrix}u\\ c\end{pmatrix}=0\iff u=-c\,d_0(z,\xi)^{-1}v(z,\xi),$$
which immediately yields the claim. 
\end{proof}

\begin{proposition}
If $\pi:S^*\pO\to \Omega$ denotes the canonical projection, then 
 $$\mathrm{ind}_{S^*\pO}\sigma_\partial(\overline\frakD)\not\in\pi^*K(\pO),$$
i.e., the Atiyah-Bott obstruction does not vanish for $\frakD$. 
\end{proposition}

In order to show this result we need to verify that the kernel-bundle $E:=\mathrm{ker}\,\wt{d}_0$  is not stably 
isomorphic to the pull-back under $\pi$ of a bundle on $\pO=S^3$. Since vector bundles on the 3-sphere are always 
stably trivial, we only have to show that $E$ is not stably trivial. 

To this end let $z_0=(1,0)\in \pO$ be fixed and let $E_0$ denote the restriction of $E$ to 
 $$S^*_{z_0}\pO=\{\xi\in\cz^2\mid (z_0,\xi)\in S^*D\}=\{\xi\in\cz^2\mid |\xi|=1,\;\re\xi_1=0\}\cong S^2.$$
We shall verify that $E_0$ is isomorphic to the Bott generator bundle on $S^2$, hence is not stably trivial; 
consequently, also $E$ cannot be.    

In fact, write $S^*_{z_0}\pO$ as the union $S_+\cup S_-$ of the upper and lower semi-sphere, 
$S_\pm=\{\xi\in S^*_{z_0}\pO\mid 0\le\pm\im\,\xi_1\le1\}$. Specializing the above Lemma \ref{lem:basis} to the case 
$z=z_0$, and noting that then $z\cdot\xi=\overline{\xi}_1$ and $z\cdot\xi^\perp=\xi_2$, we find that 
\begin{align*}
 s_+(\xi)&=(0,1),\qquad \xi\in S_+,\\
 s_-(\xi)&=\Big(-\phi(z_0,\xi)\overline{\xi}_1 e^{-ir/\overline{\xi}_1},\xi_2\Big),\qquad \xi\in S_-,
\end{align*}
define two non-vanishing sections of $E_0$ over $S_+$ and $S_-$, respectively. Note that $s_-(\xi)=(0,\xi_2)$ 
near the equator $\{\xi=(0,\xi_2)\mid |\xi_2|=1\}\cong S^1$. In other words, the bundle $E_0$ is obtained 
by clutching together the trivial one-dimensional bundles over $S_+$ and $S_-$, respectively, via the clutching function 
$f:S^1\to\cz\setminus\{0\},\,f(\xi_2)=\xi_2$. Thus $E_0$ coincides with the Bott generator.

\section{Generalized pseudodifferential operator algebras}\label{sec:04}

The aim of this section is to introduce an abstract framework in which 
principal facts and techniques known from the theory of pseudodifferential operators $($on manifolds with and 
without boundary and also on manifolds with singularities$)$ can be formalized. 
We begin with two examples to motivate this formalization: 

\begin{example}\label{ex:algebra-ex1}
Let $M$ be a smooth closed Riemannian manifold. We denote by $G$ the set of all $g=(M,F)$, 
where $F$ is a smooth Hermitian vector bundle over $M$. Let   
 $$H(g):= L^2(M,F), \qquad g=(M,F),$$
be the Hilbert space of square integrable sections of $F$. 
If $\frakg=(g^0,g^1)$ with $g^j=(M,F_j)$ we let 
 $$L^\mu(\frakg):=L^{\mu}_\cl(M;F_0,F_1),\qquad \mu\le0,$$ 
denote the space of classical pseudodifferential operators of order $\mu$ acting from $L_2$-sections of $F_0$ to 
$L_2$-sections of $F_1$. Note that there is a natural identification  
 $$L^{\mu}_\cl(M;F_0\oplus F_1,F^\prime_0\oplus F^\prime_1) 
     =\left\{
     \begin{pmatrix}A_{00}&A_{01}\\A_{10}&A_{11}\end{pmatrix}
     \;\Bigg\vert\; A_{ij}\in L^\mu(M;F_j,F_i^\prime)\right\}.$$    
With $\pi:S^*M\to M$ being the canonical projection of the co-sphere bundle 
to the base, we let 
 $$E(g):=\pi^*F, \qquad g=(M,F),$$
and then for $A\in L^0(\frakg)$, $\frakg=(g^0,g^1)$, the usual principal symbol is a map 
 $$\sigma_\psi^0(A):E(g^0)\lra E(g^1);$$
it vanishes for operators of negative order. Obviously we can compose operators $($only$)$ if the bundles they 
act in fit together and, in this case, the principal symbol behaves multiplicatively. Taking the $L_2$-adjoint 
induces a map $L^{\mu}_\cl(M;F_0,F_1)\to L^{\mu}_\cl(M;F_1,F_0)$ well-behaved with the principal symbol, i.e., 
$\sigma_\psi^0(A^*)=\sigma_\psi^0(A)^*$, where the $*$ on the right indicates the adjoint morphism 
$($obtained by passing fibrewise to the adjoint$)$. 
\end{example}

\begin{example}\label{ex:algebra-ex2}
Considerations analogous to that of Example $\mathrm{\ref{ex:algebra-ex1}}$ apply to Boutet de Monvel's 
algebra for manifolds with smooth boundary. Here the weights are $g=(\Omega,E,F)$, where $E$ and $F$ are 
Hermitian bundles over $\Omega$ and $\pO$, respectively, and 
\begin{align*}
 L^\mu(\frakg)&=\calB^{\mu,0}(\Omega;(E_0,F_0),(E_1,F_1)),\qquad 0\le -\mu\in\gz,\\
 H(g)&=L^2(\Omega,E)\oplus L^2(\pO,F),
\end{align*}
for $g=(\Omega,E,F)$ and $\frakg=(g_0,g_1)$ with $g_j=(\Omega,E_j,F_j)$. The principal symbol has two 
components, $\sigma^\mu(\calA)=(\sigma^\mu_\psi(\calA),\sigma^\mu_\partial(\calA))$. 
\end{example}

\subsection{The general setup}\label{sec:04.1} 

Let $G$ be a set; the elements of $G$ we 
will refer to as \emph{weights}. With every weight $g\in G$ there is associated a Hilbert space $H(g)$. 
There is a weight such that $\{0\}$ is the associated Hilbert space. 
With any $\frakg=(g^0,g^1)\in G\times G$ there belong vector-spaces of operators 
 $$L^{-\infty}(\frakg)\subset L^0(\frakg)\subset\scrL(H(g^0),H(g^1));$$
$0$ and $-\infty$ we shall refer to as the \emph{order} of the operators, those of order $-\infty$ we shall also call 
\emph{smoothing operators}. We shall assume that smoothing operators induce compact operators in the 
corresponding Hilbert spaces and that the identity operator is an element of $L^0(\frakg)$ for any pair 
$\frakg=(g,g)$. 

\begin{remark}
Let us point out that in this abstract setup the operators have order at most $0$. 
This originates from the fact that in 
applications the use of order-reductions often allows to reduce general order situations to the zero order case 
$($see for example the corresponding reduction in the proof of Theorem $\mathrm{\ref{thm:complex-main}})$.  
\end{remark}

Two pairs $\frakg_0$ and $\frakg_1$ are called \emph{composable} if $\frakg_0=(g^0,g^1)$ and 
$\frakg_1=(g^1,g^2)$, and in this case we define 
 $$\frakg_1\circ\frakg_0=(g^0,g^2).$$
We then request that composition of operators induces maps 
 $$L^\mu(\frakg_1)\times L^\nu(\frakg_0)\lra L^{\mu+\nu}(\frakg_1\circ\frakg_0),\qquad \mu,\nu\in\{-\infty,0\},$$
whenever the involved pairs of weights are composable. 

\begin{definition}
With the previously introduced notation let  
 $$L^{\bullet}=\mathop{\mbox{\LARGE$\cup$}}_{\frakg\in G\times G} L^0(\frakg).$$
By abuse of language, we shall speak of the \emph{algebra} $L^\bullet$.  
\end{definition}

For a pair of weights $\frakg=(g^0,g^1)$ its \emph{inverse pair} is defined as $\frakg^{(-1)}=(g^1,g^0)$. We 
shall assume that $L^\bullet$ is \emph{closed under taking adjoints}, i.e., if $A\in L^\mu(\frakg)$ then the 
adjoint of $A:H(g^0)\to H(g^1)$ is realized by an operator $A^*\in L^\mu(\frakg^{(-1)})$. 

\begin{definition}
Let $A\in L^0(\frakg)$. Then $B\in L^0(\frakg^{(-1)})$ is called a \emph{parametrix} of $A$ if 
$AB-1\in L^{-\infty}(\frakg\circ\frakg^{(-1)})$ and $BA-1\in L^{-\infty}(\frakg^{(-1)}\circ\frakg)$. 
\end{definition}

In other words, a parametrix is an inverse modulo smoothing operators. 

\subsubsection{The Fredholm property}\label{sec:04.1.0}

It is clear that if $A\in L^0(\frakg)$ has a parametrix then $A$ induces a Fredholm operator in the corresponding 
Hilbert spaces. 

\begin{definition}
We say that $L^\bullet$ has the \emph{Fredholm property} if, for every $A\in L^0(\frakg)$, $\frakg=(g^0,g^1)$, 
the following holds true$:$ 
 $$\text{$A$ has a parametrix }\iff \text{$A:H(g^0)\to H(g^1)$ is a Fredholm operator}.$$  
\end{definition}

It is well-known that Boutet de Monvel's algebra has the Fredholm property, see Theorem 7 in Section 3.1.1.1 of 
\cite{ReSc}, for example.  

\subsubsection{The block-matrix property}\label{sec:04.1.1}

We shall say that the algebra $L^\bullet$ has the \emph{block-matrix property} if there exists a map 
 $$(g^0,g^1)\mapsto g^0\oplus g^1:G\times G\to G$$
which is associative, i.e., 
$(g^0\oplus g^1)\oplus g^2=g^0\oplus (g^1\oplus g^2)$, such that 
 $$H(g^0\oplus g^1)=H(g^0)\oplus H(g^1), \qquad (g^0,g^1)\in G\times G,$$
and such that 
 $$L^\mu(\frakg),\qquad \frakg=(g_0^0\oplus\ldots\oplus g_{\ell}^0,g_0^1\oplus\ldots\oplus g_{k}^1),$$
can be identified with the space of $(k+1)\times(\ell+1)$-matrices  
\begin{equation*}
 \begin{pmatrix}A_{00}& \cdots & A_{0\ell}\\ \vdots & & \vdots \\A_{k0}&\cdots& A_{k\ell}\end{pmatrix}
     \colon
     \begin{matrix}H(g^0_0)\\ \oplus\\ \vdots \\ \oplus \\ H(g^0_\ell)\end{matrix}
     \lra
     \begin{matrix}H(g^0_1)\\ \oplus\\ \vdots \\ \oplus \\ H(g^1_k)\end{matrix},
     \qquad A_{ij}\in L^\mu((g_j^0,g_i^1)).
\end{equation*}

\subsubsection{Classical algebras and principal symbol map}\label{sec:04.1.2}

An algebra $L^\bullet$ will be called \emph{classical}, and then for clarity denoted by $L^\bullet_\cl$, 
if there exists a map, called \emph{principal symbol} map, 
 $$A\mapsto\sigma(A)=\big(\sigma_1(A),\ldots,\sigma_n(A)\big)$$
assigning to each $A\in L^0_\cl(\frakg)$, $\frakg=(g^0,g^1)$,  an $n$-tuple of bundle morphisms 
 $$\sigma_\ell(A):\;E_\ell(g^0)\lra E_\ell(g^1)$$
between $($finite or infinite dimensional$)$ Hilbert space bundles $E_\ell(g^j)$ over some base $B_\ell(g^j)$, 
such that the following properties are valid:  
\begin{itemize}
 \item[$(i)$] The map is linear, i.e., 
   $$\sigma(A+B)=\sigma(A)+\sigma(B)
     :=\big(\sigma_1(A)+\sigma_1(B),\ldots,\sigma_n(A)+\sigma_n(B)\big)$$
  whenever $A,B\in L^0(\frakg)$. 
 \item[$(ii)$] The map respects the composition of operators, i.e., 
   $$\sigma(BA)=\sigma(B)\sigma(A)
     :=\big(\sigma_1(B)\sigma_1(A),\ldots,\sigma_n(B)\sigma_n(A)\big)$$
  whenever $A\in L^0(\frakg_0)$ and $B\in L^0(\frakg_1)$ with composable pairs $\frakg_0$ and $\frakg_1$. 
 \item[$(iii)$] The map is well-behaved with the adjoint, i.e., for any $\ell$, 
   $$\sigma_\ell(A^*)=\sigma_\ell(A)^*:\;E_1^\ell(g_1)\lra E_0^\ell(g_0),$$  
  where $\sigma_\ell(A)^*$ denotes the adjoint morphism $($obtained by taking fibrewise the adjoint$)$; 
  for brevity, we shall also write $\sigma(A^*)=\sigma(A)^*$.  
 \item[$(iv)$] $\sigma(R)=(0,\ldots,0)$ for every smoothing operator $R$. 
\end{itemize}

\begin{definition}
$A\in L^0_\cl(\frakg)$ is called \emph{elliptic} if its principal symbol $\sigma(A)$ is invertible, i.e., 
all bundle morphisms $\sigma_1(A),\ldots,\sigma_n(A)$ are isomorphisms. 
\end{definition}  

Besides the above properties $(i)-(iv)$ we shall assume 
\begin{itemize}
 \item[$(v)$] $A\in L^0_\cl(\frakg)$ is elliptic if, and only if, it has a parametrix $B\in L^0_\cl(\frakg^{(-1)})$. 
\end{itemize}

Finally, in case $L^\bullet_\cl$ has the block-matrix property, we shall also assume that the identification with 
block-matrices from Section \ref{sec:02.2} has an analogue on the level of principal symbols. 

\subsection{Operators of Toeplitz type}\label{sec:04.2}

In the following let $\frakg=(g^0,g^1)$ and $\frakg_j=(g^j,g^j)$ for $j=0,1$. Let $P_j\in L^0(\frakg_j)$ be 
projections, i.e., $P_j^2=P_j$. We then define, 
for $\mu=0$ or $\mu=-\infty$,  
\begin{align*}
 T^\mu(\frakg;P_0,P_1):=&\Big\{A\in L^\mu(\frakg)\mid (1-P_1)A=0,\;A(1-P_0)=0\Big\}\\
 =&\Big\{P_1{A^\prime}P_0 \mid {A^\prime}\in L^\mu(\frakg)\Big\}. 
\end{align*}
If we set 
\begin{equation*}
 H(g_j,P_j):=\mathrm{im}\,P_j=P_j\big(H(g_j)\big), 
\end{equation*}
then $H(g_j,P_j)$ is a closed subspace of $H(g_j)$ and we have the inclusions  
 $$ T^{-\infty}(\frakg,P_0,P_1)\subset T^0(\frakg,P_0,P_1)\subset \scrL\big(H(g_0,P_0),H(g_1,P_1)\big).$$
Clearly, smoothing operators are not only bounded but again compact. 

The union of all these spaces (i.e., involving all weights and projections) we shall denote by $T^\bullet$. 
We shall call $T^\bullet$ a \emph{Toeplitz algebra} and refer to elements of $T^\bullet$ as 
\emph{Toeplitz type operators}.    

\begin{definition}\label{def:param-t}
Let $A\in T^0(\frakg;P_0,P_1)$. Then $B\in T^0(\frakg^{(-1)};P_1,P_0)$ is called a \emph{parametrix} of $A$ if 
 $$AB-P_1\in T^{-\infty}(\frakg\circ\frakg^{(-1)};P_1,P_1), \qquad 
     BA-P_0\in T^{-\infty}(\frakg^{(-1)}\circ\frakg;P_0,P_0).$$  
\end{definition}

\subsubsection{Classical operators and principal symbol}\label{sec:04.2.1}

The previous definitions extend, in an obvious way, to classical algebras; again we shall use the subscript $\cl$ to 
indicate this, i.e., we write $T^\bullet_\cl$. We associate with $A\in T^0_\cl(\frakg;P_0,P_1)$ a principal symbol in 
the following way: Since the $P_j$ are projections, the associated symbols $\sigma_\ell(P_j)$ are projections in the 
bundles $E_\ell(g^j)$ and thus define subbundles 
 $$E_\ell (g^j,P_j):=\mathrm{im}\,\sigma_\ell(P_j)=\sigma_\ell(P_j)(E_\ell(g^j)).$$
For $A\in T^0_\cl(\frakg;P_0,P_1)$ we then define 
 $$\sigma(A;P_0,P_1)=\big(\sigma_1(A;P_0,P_1),\ldots,\sigma_n(A;P_0,P_1)\big)$$
by 
\begin{equation}\label{eq:symb-t}
 \sigma_\ell(A;P_0,P_1)=\sigma_\ell(A):\;E_\ell (g^0,P_0)\lra E_\ell (g^1,P_1);
\end{equation}
note that $\sigma_\ell(A)$ maps into $E_\ell (g^1,P_1)$ in view of the fact that $(1-P_1)A=0$. 

\begin{remark}
The principal symbol map defined this way satisfies the obvious analogues of properties $(i)$, $(ii)$, and $(iv)$ from 
Section $\mathrm{\ref{sec:04.1.2}}$. Concerning property $(iii)$ of the adjoint, observe that there is a natural identification 
of the dual of $H(g,P)$ with the space $H(g,P^*)$. This leads to maps 
 $$A\mapsto A^*\colon T^\mu_{(\cl)}(\frakg;P_0,P_1)\lra T^\mu_{(\cl)}(\frakg^{(-1)};P_1^*,P_0^*),$$
and $(iii)$ generalizes correspondingly.  
\end{remark}

\begin{definition}\label{def:ell-t}
An operator $A\in T^0_\cl(\frakg;P_0,P_1)$ is called \emph{elliptic} if its principal symbol $\sigma(A;P_0,P_1)$ 
is invertible, i.e., all bundle morphisms $\sigma_1(A;P_0,P_1),\ldots,$ $\sigma_n(A;P_0,P_1)$ from 
\eqref{eq:symb-t} are isomorphisms. 
\end{definition}  

Property $(v)$ from Section $\mathrm{\ref{sec:04.1.2}}$, whose validity was a mere assumption for the algebra 
$L^\bullet_\cl$, can be shown to remain true for the Toepltz algebra $T^\bullet_\cl$, see 
Theorem 3.12 of \cite{Seil}: 

\begin{theorem}\label{thm:seil2}
For $A\in T^0_\cl(\frakg;P_0,P_1)$ the following properties are equivalent$:$ 
\begin{itemize}
 \item[a$)$] $A$ is elliptic $($in the sense of Definition $\mathrm{\ref{def:ell-t}})$, 
 \item[b$)$] $A$ has a parametrix $($in the sense of Definition $\mathrm{\ref{def:param-t}})$. 
\end{itemize} 
\end{theorem}

Similarly, the Fredholm property in $L^\bullet$ is inherited by the respective Toeplitz algebra $T^\bullet$, as 
has been shown in Theorem 3.7 of \cite{Seil}$:$

\begin{theorem}\label{thm:seil1}
Let $L^\bullet$ have the Fredholm property. 
For $A\in T^0(\frakg;P_0,P_1)$ the following properties are equivalent$:$ 
\begin{itemize}
 \item[a$)$] $A$ has a parametrix $($in the sense of Definition $\mathrm{\ref{def:param-t}})$. 
 \item[b$)$] $A:H(g^0,P_0)\to H(g^1,P_1)$ is a Fredholm operator. 
\end{itemize}
\end{theorem}

\section{Complexes in operator algebras}\label{sec:05} 

In this section we study complexes whose single operators belong to a general algebra $L^\bullet$. So let 
\begin{equation}\label{eq:L-complex}
 \frakA:\;\dots \xrightarrow{A_{-1}} H(g^0)\xrightarrow{A_0}H(g^1)\xrightarrow{A_1}H(g^2)
 \xrightarrow{A_2}H(g^3)\xrightarrow{A_3}\ldots,
\end{equation}
be a complex with operators $A_j\in L^0(\frakg_j)$, $\frakg_j=(g^j,g^{j+1})$. Of course, $\frakA$ is also 
a Hilbert space complex in the sense of Section \ref{sec:02}. 
Note that the Laplacians associated with $\frakA$ satisfy 
$\Delta_j\in L^0((g^j,g^j))$, $j\in\gz$. 

\subsection{Fredholm complexes and parametrices}\label{sec:05.1} 

The notion of parametrix of a Hilbert space complex has been given in Definition \ref{def:paramterix}. 
In the context of operator algebras the definition is as follows.  

\begin{definition}\label{def:param_in_L}
A \emph{parametrix in $L^\bullet$} of the complex $\frakA$ is a sequence of operators 
$B_j\in L^0(\frakg_j^{(-1)})$, $j\in\gz$, such that 
 $$A_{j-1}B_{j-1}+B_jA_j-1\;\in\; L^{-\infty}((g^j,g^j)),\qquad j\in\gz.$$
In case $B_jB_{j+1}=0$ for every $j$ we call such a parametrix a complex. 
\end{definition}

Clearly, a parametrix in $L^\bullet$ is also a parametrix in the sense of Definition \ref{def:paramterix}, 
but not vice versa.   

\begin{proposition}\label{prop:L-complex-param}
Let $L^\bullet$ have the Fredholm property. 
Then $\frakA$ is a Fredholm complex if, and only if, $\frakA$ has a parametrix in $L^\bullet$. 
\end{proposition}
\begin{proof}
If $\frakA$ has a parametrix it is a Fredholm complex by Theorem \ref{thm:complex-basics}. 
Vice versa, the Fredholmness of $\frakA$ is equivalent 
to the simultaneous Fredholmness of all Laplacians $\Delta_j$. By assumption on $L^\bullet$, this in 
turn is equivalent to the existence of parametrices $D_j\in L^0((g^j,g^j))$ to $\Delta_j$ for every $j$. Then 
$B_j:=D_jA_j^*$ is a parametrix in $L^\bullet$. In fact, the identity $A_j\Delta_j=\Delta_{j+1}A_j$ implies 
that $D_{j+1}A_j\equiv A_jD_j$, where $\equiv$ means equality modulo smoothing operators. Therefore, 
\begin{align*}
  B_jA_j+A_{j-1}B_{j-1}&=D_jA_j^*A_j+A_{j-1}D_{j-1}A_{j-1}^*\\
  &\equiv D_jA_j^*A_j+D_{j}A_{j-1}A_{j-1}^*=D_j\Delta_j\equiv1.
\end{align*}
This finishes the proof. 
\end{proof}

The parametrix constructed in the previous definition is, in general, not a complex. To assure the existence of a 
parametrix that is also a complex one needs to pose an additional condition on $L^\bullet$ $($as discussed below, 
it is a mild condition, typically satisfied in applications$)$. 
 
\begin{definition}\label{def:fred-ext}
$L^\bullet$ is said to have the \emph{extended Fredholm property} if it has the Fredholm property and for every 
$A\in L^0(\frakg)$, $\frakg=(g,g)$, with $A=A^*$ and which is a Fredholm operator in $H(g)$, 
there exists a parametrix $B\in L^0(\frakg)$ such that 
 $$AB=BA=1-\pi$$
with $\pi\in\scrL(H(g))$ being the orthogonal projection onto $\mathrm{ker}\,A$. 
\end{definition}

Note that, with $A\in L^0(\frakg)$ and $\pi$ as in the previous definition, we have the orthogonal decomposition 
$H(g)=\mathrm{im}\,A\oplus\mathrm{ker}\,A$ and $A:\mathrm{im}\,A\to\mathrm{im}\,A$ is an isomorphism. 
If $T$ denotes the inverse of this isomorphism, then the condition of Definition \ref{def:fred-ext} can be rephrased 
as follows$:$ It is asked that there exists a $B\in L^0(\frakg)$ with 
\begin{equation}\label{eq:spezielle-param}
 Bu=T(1-\pi)u\qquad\text{for all }u\in H(g).
\end{equation}  

\begin{theorem}\label{thm:L-complex-param}
Let $L^\bullet$ have the extended Fredholm property. 
Then $\frakA$ is a Fredholm complex if, and only if, $\frakA$ has a parametrix in $L^\bullet$ which is a complex. 
\end{theorem} 
\begin{proof}
Let $\frakA$ be a Fredholm complex. By assumption, there exist parametrices $D_j\in L^0((g^j,g^j))$ of 
the complex Laplacians $\Delta_j$ with $\Delta_jD_j=D_j\Delta_j=1-\pi_j$, where 
$\pi_j\in\scrL(H(g^j))$ is the orthogonal projection onto the kernel of $\Delta_j$. Now define $B_j:=D_jA_j^*$. 
As we have shown in the proof of Proposition \ref{prop:L-complex-param}, the $B_j$ define a parametrix. 
Since $D_{j+1}$ maps $\mathrm{im}\,A_{j+1}^*=(\mathrm{ker}\,A_{j+1})^\perp$ into itself, and 
$\mathrm{im}\,A_{j+1}^*\subset\mathrm{ker}\,A_{j}^*$, we obtain 
$A_j^*D_{j+1}A_{j+1}^*=0$, hence $B_{j}B_{j+1}=0$. 
\end{proof}

The following theorem gives sufficient conditions for the validity of the extended Fredholm property. 

\begin{theorem}\label{thm:sandwich}
Let $L^\bullet$ have the Fredholm property and assume the following$:$ 
\begin{itemize}
 \item[a$)$] If $A=A^*\in L^0(\frakg)$, $\frakg=(g,g)$, is a Fredholm operator in $H(g)$, 
  then the orthogonal projection onto the kernel of $A$ is an element of $L^{-\infty}(\frakg)$. 
 \item[b$)$] $R_1TR_0\in L^{-\infty}(\frakg)$, $\frakg=(g,g)$, whenever $R_0,R_1\in L^{-\infty}(\frakg)$ and 
  $T\in\scrL(H(g))$. 
\end{itemize}
Then $L^\bullet$ has the extended Fredholm property. 
\end{theorem} 

In other words, condition $b)$ asks that sandwiching a bounded operator $T$ $($not necessarily belonging to 
the algebra$)$ between two smoothing operators always results in being a smoothing operator. 
A typical example are pseudodifferential operators on closed manifolds, where the smoothing operators are 
those integral operators with a smooth kernel, and sandwiching any operator which is continuous in $L_2$-spaces 
results again in an integral operator with smooth kernel. Similarly, also Boutet de Monvel's algebra and many other 
algebras of pseudodifferential operators are covered by this theorem.  

\begin{proof}[Proof of Theorem $\mathrm{\ref{thm:sandwich}}$]
Let $A=A^*\in L^0(\frakg)$, $\frakg=(g,g)$, be a Fredholm operator in $H(g)$. Let $B=T(1-\pi)\in\scrL(H(g))$ 
be as in \eqref{eq:spezielle-param}; initially, $B$ is only a bounded operator in $H(g)$, but we shall 
show now that $B$ infact belongs to $L^0(\frakg)$. 

By assumption we find a parametrix $P\in L^0(\frakg)$ to $A$, i.e. $R_1:=1-PA$ and $R_0:=1-AP$ belong to 
$L^{-\infty}(\frakg)$. Then, on $H(g)$, 
\begin{align*}
 B-P&=(PA+R_1)(P-B)=P(\pi-R_0)+R_1(P-B),\\
 B-P&=(P-B)(AP+R_0)=(\pi-R_1)P+(P-B)R_0
\end{align*}
Substituting the second equation into the first and rearranging terms yields 
 $$B-P=P(\pi-R_0)+R_1(\pi-R_1)P+R_1(P-B)R_0.$$
The right-hand side belongs to $L^{-\infty}(\frakg)$ by assumptions $(a)$ and $(b)$. Since $P$ belongs to 
$L^0(\frakg)$, then so does $B$. 
\end{proof}

\subsection{Elliptic complexes}\label{sec:05.2} 

Let us now assume that we deal with a classical algebra $L^\bullet_\cl$ and the complex $\frakA$ from 
\eqref{eq:L-complex} is made up of operators $A_j\in L^0_\cl(\frakg_j)$, $\frakg_j=(g^j,g^{j+1})$. 
If $A\mapsto\sigma(A)=\big(\sigma_1(A),\ldots,\sigma_n(A)\big)$ is the associated principal symbol map, 
cf.\  Section \ref{sec:04.1.2}, then we may associate with $\frakA$ the families of complexes 
\begin{equation}\label{eq:symbol-complex}
   \sigma_\ell(\frakA)\colon\ldots
   \xrightarrow{\sigma_\ell(A_{-1})} E_\ell(g^0)
   \xrightarrow{\sigma_\ell(A_0)}E_\ell(g^1)
   \xrightarrow{\sigma_\ell(A_1)}E_\ell(g^2)\xrightarrow{\sigma_\ell(A_2)}\ldots,
\end{equation}
for $\ell=1,\ldots,n$; here we shall assume that, for each $\ell$, all bundles $E^\ell(g)$, $g\in G$, have 
the same base space and that $\sigma_\ell(\frakA)$ is a family of complexes as described in 
Section \ref{sec:02.3}. 

\begin{definition}\label{def:L-complex-elliptic}
The complex $\frakA$ in $L^\bullet_\cl$ is called \emph{elliptic} if all the associated families of 
complexes $\sigma_\ell(\frakA)$, $\ell=1,\ldots,n$, are exact families $($in the sense of 
Section $\mathrm{\ref{sec:02.3}})$. 
\end{definition}

\begin{theorem}\label{thm:L-complex-classical}
For a complex $\frakA$ in $L^\bullet_\cl$ the following properties are equivalent$:$
\begin{itemize}
 \item[a$)$] $\frakA$ is elliptic. 
 \item[b$)$] All Laplacians $\Delta_j$, $j\in\gz$, associated with $\frakA$ are elliptic. 
\end{itemize}
These properties imply 
\begin{itemize}
 \item[c$)$]  $\frakA$ has a parametrix in $L^\bullet_\cl$. 
 \item[d$)$]  $\frakA$ is a Fredholm complex. 
\end{itemize}
In case $L^\bullet_\cl$ has the Fredholm property, all four properties are equivalent. 
In presence of the extended Fredholm property, the parametrix can be chosen to be a complex. 
\end{theorem} 
\begin{proof}
The equivalence of a$)$ and b$)$ is simply due to the fact that the principal symbol 
$\sigma_\ell(\Delta_j)$ just coincides with the $j$-th Laplacian associated with $\sigma_\ell(\frakA)$ 
and therefore simultaneous exactness of $\sigma_\ell(\frakA)$, $1\le\ell\le n$, 
 in the $j$-th position is equivalent to the invertibility of all $\sigma_\ell(\Delta_j)$, i.e., the ellipticity of $\Delta_j$. 
The rest is seen as above in Poposition \ref{prop:L-complex-param} and Theorem \ref{thm:L-complex-param}.
\end{proof}

The complex $\frakA$ induces the families of complexes $\sigma_\ell(\frakA)$. The following theorem is a kind of 
reverse statement, i.e., starting from exact families of complexes we may construct a complex of operators. 
For a corresponding result in the framework of Boutet de Monvel's algebra see Lemma 1.3.10 in \cite{PiSc} and 
Theorem 8.1 in \cite{KTT}. 

\begin{theorem}\label{thm:L-lift1}
Assume that $L^\bullet_\cl$ has the extended Fredholm property. Let $N\in\nz$ and 
$A_j\in L^0(\frakg_j)$, $\frakg_j=(g^j,g^{j+1})$, $j=0,\ldots, N$, be such that  
the associated sequences of principal symbols form exact families of complexes 
\begin{equation*}
 \;0\lra E_\ell(g^0)\xrightarrow{\sigma_\ell(A_0)}E_\ell(g^1)
  \xrightarrow{\sigma_\ell(A_1)}E_\ell(g^2)\ldots \xrightarrow{\sigma_\ell(A_N)}E_\ell(g^{N+1})\lra 0, 
\end{equation*}
$\ell=1,\ldots,n$. Then there exist operators $\wt{A}_j\in L^{0}(\frakg_j)$, $j=0,\ldots,N$, 
with $\sigma(\wt{A}_j)=\sigma(A_j)$ and 
such that 
 $$\wt{\frakA}:\;0\lra H(g^0)\xrightarrow{\wt{A}_0}H(g^1)\xrightarrow{\wt{A}_1}H(g^2)
     \ldots\xrightarrow{\wt{A}_N}H(g^{N+1})\lra0$$
is a complex. In case $A_{j+1}A_{j}$ is smoothing for every $j$, the operators $\wt{A}_j$ can be chosen 
in such a way that $\wt{A}_j-A_j$ is smoothing for every $j$. 
\end{theorem}
\begin{proof}
We take $\wt{A}_N=A_N$ and then apply an iterative procedure, first modifying the operator 
$A_{N-1}$ and then, subsequently, the operators $A_{N-2},\ldots,A_0$. 

Consider the Laplacian 
$\Delta_{N+1}=A_NA_N^*\in L^0_\cl(g^{N+1},g^{N+1})$. Since, by assumption, any $\sigma_\ell(A_N)$ is  
$($fibrewise$)$ surjective, $\sigma_\ell(\Delta_{N+1})=\sigma_\ell(A_N)\sigma_\ell(A_N)^*$ is an isomorphism. 
Hence $\Delta_{N+1}$ is elliptic. By the extended Fredholm property we find a parametrix 
$D_{N+1}\in L^0((g^{N+1},g^{N+1}))$ of $\Delta_{N+1}$ with 
$\Delta_{N+1}D_{N+1}=D_{N+1}\Delta_{N+1}=1-\pi_{N+1}$, where 
$\pi_{N+1}\in L^{-\infty}(g^{N+1},g^{N+1})$ is the orthogonal projection in $H(g^{N+1})$ onto the kernel of 
$\Delta_{N+1}$, i.e., onto the kernel of $A_{N}^*$. Then it is straightforward to check that 
$\Pi_N:=1-A_N^*D_{N+1}A_N$ is the orthogonal projection in $H(g^{N})$ onto the kernel of $A_N$. 
Then let us set 
 $$\wt{A}_{N-1}:=\Pi_NA_{N-1}=A_{N-1}+R_{N-1},\qquad R_{N-1}=-A_N^*D_{N+1}A_NA_{N-1}.$$
Since $\sigma(A_NA_{N-1})=\sigma(A_N)\sigma(A_{N-1})=0$ we find that $\sigma(R_{N-1})=0$. Obviously, 
if $A_NA_{N-1}$ is smoothing then so is $R_{N-1}$. This finishes the first step of the procedure. 

Next we are going to modify $A_{N-2}$. For notational convenience redefine $A_{N-1}$ as $\wt{A}_{N-1}$. 
Similarly as above, the $n$-th Laplacian $\Delta_N=A_N^*A_N+\wt{A}_{N-1}\wt{A}_{N-1}^*$ is elliptic, 
due to the exactness of the symbol complexes. We then let $D_N$ be a parametrix with  
$\Delta_{N}D_{N}=D_{N}\Delta_{N}=1-\pi_{N}$, where $\pi_{N}\in L^{-\infty}(g^{N},g^{N})$ is the orthogonal 
projection in $H(g^{N})$ onto the kernel of $\Delta_{N}$. Then set $\Pi_{N-1}=1-A_{N-1}^*D_{N}A_{N-1}$. 
Now observe that 
\begin{align*}
 (1-\Pi_{N-1})^2&=A_{N-1}^*D_NA_{N-1}A_{N-1}^*D_NA_{N-1}\\
 &=A_{N-1}^*D_N\Delta_N D_NA_{N-1}A_{N-1}-A_{N-1}^*D_NA_{N}^*A_{N}D_NA_{N-1}\\
 &=1-\Pi_{N-1},
\end{align*}
since $D_N\Delta_N D_N=D_N(1-\pi_N)=D_N$ and $D_N$ maps $\mathrm{im}\,(A_{N-1})$ into itself, hence 
the second summand vanishes in view of $\mathrm{im}\,(A_{N-1})\subset\mathrm{ker}\,(A_{N})$. 
Similarly one verifies that $\mathrm{im}\,(\Pi_{N-1})=\mathrm{ker}\,(A_{N-1})$. In other words, 
$\Pi_{N-1}$ is the orthogonal projection in $H(g^{N-1})$ onto the kernel of $A_{N-1}$. Then proceed as 
above, setting $\wt{A}_{N-2}=\Pi_{N-1}\wt{A}_{N-1}$. 
Repeat this step for $A_{N-3}$, and so on. 
\end{proof}

\begin{remark}
Let notations and assumptions be as in Theorem $\mathrm{\ref{thm:L-lift1}}$. Though the $A_{j}$ do not form a 
complex, the compostions $A_{j+1}A_j$ have vanishing principal symbols and thus can be considered as ``small". 
In the literature such kind of almost-complexes are known as \emph{essential complexes}, cf.\ \cite{AmVa}, 
or \emph{quasicomplexes}, cf.\ \cite{KTT}. In this spirit, Theorem $\mathrm{\ref{thm:L-lift1}}$ says that 
any elliptic quasicomplex in $L^\bullet_\cl$ can be ``lifted" to an elliptic complex. 
\end{remark}

\section{Complexes in Toeplitz algebras}\label{sec:06} 

After having developed the theory for complexes in an operator algebra 
$L^\bullet_{(\cl)}$, let us now turn to complexes in the associated Toeplitz algebra $T^\bullet_{(\cl)}$. 
These have the form 
\begin{equation}\label{eq:T-complex}
 \frakA_\frakP:\;\dots \xrightarrow{A_{-1}} H(g^0;P_0)\xrightarrow{A_0}H(g^1;P_1)\xrightarrow{A_1}H(g^2;P_2)
 \xrightarrow{A_2}\ldots,
\end{equation}
with operators $A_j\in L^0_{(\cl)}(\frakg_j;P_j;P_{j+1})$, $\frakg_j=(g^j,g^{j+1})$; we use the subscript 
$\frakP$ to indicate the involved sequence of projections $P_j$, $j\in\gz$. Of course, if all projections are equal to the 
identity, we obtain a usual complex in $L^\bullet$. 

As we shall see, the basic definitions used for complexes in $L^\bullet$ generalize straightforwardly to the 
Toeplitz case. However, the techniques developed in the previous section do not apply directly to complexes in 
Toeplitz algebras. Mainly, this is due to the 
fact that Toeplitz algebras behave differently under application of the adjoint, i.e., 
 $$A\mapsto A^*\colon T^0(\frakg;P_0,P_1)\lra T^0(\frakg;P_1^*,P_0^*).$$
As a consequence, it is for instance not clear which operators substitute the  Laplacians that played a 
decisive role in the analysis of complexes in $L^\bullet$.   

To overcome this difficulty, we shall develop a method of lifting a complex $\frakA_\frakP$ to a complex in 
$L^\bullet$, which preserves the essential properties of $\frakA_\frakP$. To the lifted complex we apply 
the theory of complexes in $L^\bullet$ and then arive at corresponding conclusions for the original complex 
$\frakA_\frakP$. 

For clarity, let us state explicitly the definitions of parametrix and ellipticity. 

\begin{definition}\label{def:param_in_T}
A \emph{parametrix in $T^\bullet$} of the complex $\frakA_\frakP$ is a sequence of operators 
$B_j\in L^0(\frakg_j^{(-1)};P_{j+1},P_j)$, $j\in\gz$, such that 
 $$A_{j-1}B_{j-1}+B_jA_j-P_j\;\in\; L^{-\infty}((g^j,g^j);P_j,P_j)\qquad j\in\gz.$$
In case $B_jB_{j+1}=0$ for every $j$ we call such a parametrix a complex. 
\end{definition}

Let, additionally, $L^\bullet=L^\bullet_\cl$ be classical with principal symbol map
$A\mapsto\sigma(A)=(\sigma_1(A),\ldots,\sigma_n(A))$. Then we associate with $\frakA_\frakP$ the 
families of complexes 
\begin{equation}\label{eq:symbol-complex-toeplitz}
     \sigma_\ell(\frakA_\frakP)\colon\ldots E_\ell(g^0,P_0)\xrightarrow{\sigma_\ell(A_0;P_0,P_1)}E_\ell(g^1,P_1)
     \xrightarrow{\sigma_\ell(A_1;P_1,P_2)}E_\ell(g^2,P_2)\ldots,
\end{equation}
cf. \eqref{eq:symb-t}. 

\begin{definition}\label{def:ellipticity_in_T}
A complex $\frakA_\frakP$ in $T^\bullet_\cl$ is called \emph{elliptic} if all $\sigma_\ell(\frakA_\frakP)$, 
$1\le\ell\le n$, are exact families of complexes.  
\end{definition}

We shall now investigate the generalization of Proposition \ref{prop:L-complex-param} and Theorems 
\ref{thm:L-complex-param}, \ref{thm:L-complex-classical} and \ref{thm:L-lift1} to the setting of 
complexes in Toeplitz algebras. 

\subsection{Lifting of complexes}\label{sec:06.1} 

Consider an at most semi-infinite complex $\frakA_\frakP$ in $T^\bullet$, i.e., 
\begin{equation}\label{eq:T-complex2}
 \frakA_\frakP\colon 0\lra H(g^0;P_0)\xrightarrow{A_0}H(g^1;P_1)\xrightarrow{A_1}H(g^2;P_2)
 \xrightarrow{A_2} \ldots
\end{equation}
with operators $A_j\in L^0_{(\cl)}(\frakg_j;P_j;P_{j+1})$, $\frakg_j=(g^j,g^{j+1})$ for $j\ge0$. 
Moreover, assume that $L^\bullet$ has the block-matrix property described in Section \ref{sec:02.2}. 

Let us define the weights 
 $$g^{[j]}:=g^j\oplus g^{j-1}\oplus\ldots\oplus g^0\in G,\qquad j=0,1,2,\ldots.$$
Then we have 
 $$H(g^{[j]})=H(g^j)\oplus H(g^{j-1})\oplus\ldots\oplus H(g^0).$$
We then define 
 $$A_{[j]}\in L^0(\frakg_{[j]}),\qquad \frakg_{[j]}:=(g^{[j]},g^{[j+1]}),$$
by 
\begin{align}\label{eq:a_jlift}
\begin{split}
 A_{[j]}(u_j,&u_{j-1},\ldots,u_0)\\
 &=\big(A_ju_j,(1-P_j)u_j,P_{j-1}u_{j-1},(1-P_{j-2})u_{j-2},P_{j-3}u_{j-3},\ldots\big).
\end{split}
\end{align}
In other words, the block-matrix representation of $A_{[j]}$ is 
\begin{align*}
  A_{[j]}=\mathrm{diag}(A_j,0,0,0,\ldots)+\mathrm{subdiag}(1-P_j,P_{j-1},1-P_{j-2},P_{j-3},\ldots).
\end{align*}
Since $A_{j+1}A_j=0$ as well as $(1-P_{j+1})A_j=0$, it follows immediately that $A_{[j+1]}A_{[j]}=0$. 
Therefore, 
\begin{equation}\label{eq:complex_lift}
 \frakA_\frakP^\wedge:\;0\lra H(g^{[0]})\xrightarrow{A_{[0]}}H(g^{[1]})\xrightarrow{A_{[1]}}H(g^{[2]})
 \xrightarrow{A_{[2]}}H(g^{[3]})\xrightarrow{A_{[3]}}\ldots,
\end{equation} 
defines a complex in $L^\bullet$. Inserting the explicit form of $H(g^{[j]})$, this complex takes the form

\vspace*{-19ex} 

$$
 \frakA_\frakP^\wedge:\;0\lra H(g^{0})
 \xrightarrow{A_{[0]}}
  \begin{matrix}\\ \\ H(g^{1})\\ \oplus\\ H(g^0)\end{matrix}
 \xrightarrow{A_{[1]}}
  \begin{matrix}\\ \\ \\ \\ H(g^{2})\\ \oplus\\ H(g^1)\\ \oplus\\ H(g^0) \end{matrix}
 \xrightarrow{A_{[2]}}
  \begin{matrix}\\ \\ \\ \\ \\ \\ H(g^{3})\\ \oplus\\ H(g^2)\\ \oplus\\ H(g^1)\\ \oplus\\ H(g^0) \end{matrix}
 \xrightarrow{A_{[3]}}\ldots
$$ 

\begin{definition}
The complex $\frakA_\frakP^\wedge$ defined in \eqref{eq:complex_lift} is called the lift of the complex 
$\frakA_\frakP$ from \eqref{eq:T-complex2}. 
\end{definition}

\begin{proposition}\label{prop:lift}
Let $\frakA_\frakP^\wedge$ be the lift of $\frakA_\frakP$ as described above. Then 
\begin{align*}
 \mathrm{ker}\,A_{[j]}
 =&\,\mathrm{ker}\,\big(A_j:H(g^j,P_j)\to H(g^{j+1},P_{j+1})\big)\oplus \\
    &\oplus \mathrm{ker}\,P_{j-1}\oplus \mathrm{im}\,P_{j-2}\oplus \mathrm{ker}\,P_{j-3}\oplus\ldots,\\
 \mathrm{im}\,A_{[j]}
 =&\,\mathrm{im}\,\big(A_j:H(g^j,P_j)\to H(g^{j+1},P_{j+1})\big)\oplus \\
    &\oplus \mathrm{ker}\,P_{j}\oplus \mathrm{im}\,P_{j-1}\oplus \mathrm{ker}\,P_{j-2}\oplus\ldots.
\end{align*}
Here, image and kernel of the projections $P_k$ refer to the maps $P_k\in\scrL(H(g^k))$. 
Therefore, both complexes have the same cohomology spaces, 
 $$\scrH_j(\frakA_\frakP)\cong\scrH_j(\frakA_\frakP^\wedge),\qquad j=0,1,2,\ldots$$
In particular,  $\frakA_\frakP$ is a Fredholm complex or an exact complex if, and only if, its lift $\frakA_\frakP^\wedge$ is 
a Fredholm complex or an exact complex, respectively. 
\end{proposition}
\begin{proof}
Let us define the map 
 $$T_j\colon H(g^j)\to H(g^{j+1})\oplus H(g^j), \qquad T_ju=(A_ju,(1-P_j)u).$$  
Then it is clear that 
\begin{align*}
 \mathrm{ker}\,A_{[j]}
 &=\mathrm{ker}\,T_j\oplus \mathrm{ker}\,P_{j-1}\oplus \mathrm{im}\,P_{j-2}\oplus \mathrm{ker}\,P_{j-3}\oplus\ldots,\\
 \mathrm{im}\,A_{[j]}
 &=\mathrm{im}\,T_j\oplus \mathrm{im}\,P_{j-1}\oplus \mathrm{ker}\,P_{j-2}\oplus \mathrm{im}\,P_{j-3}\oplus\ldots.
\end{align*}
Now observe that $T_ju=0$ if, and only if, $u\in\mathrm{ker}\,(1-P_j)=H(g^j,P_j)$ and $A_ju=0$. This shows
 $$\mathrm{ker}\,T_j=\mathrm{ker}\,\big(A_j:H(g^j,P_j)\to H(g^{j+1},P_{j+1})\big).$$
Moreover, writing $u=v+w$ with $v\in H(g^j,P_j)$ and $w\in\mathrm{ker}\,P_j$, we obtain $T_ju=(A_jv,w)$. This shows 
 $$\mathrm{im}\,T_j=\mathrm{im}\,\big(A_j:H(g^j,P_j)\to H(g^{j+1},P_{j+1})\big)\oplus \mathrm{ker}\,P_j$$
and completes the proof. 
\end{proof}

\subsection{Fredholmness, parametrices and ellipticity of Toeplitz complexes}\label{sec:06.2} 

The next theorem shows that a parametrix of the lift induces a parametrix of the original complex. 

\begin{theorem}\label{thm:lift-param}
Let $\frakA_\frakP^\wedge$ be the lift of $\frakA_\frakP$ as described above. 
If $\frakA_\frakP^\wedge$ has a parametrix in $L^\bullet$ then 
$\frakA_\frakP$ has a parametrix in $T^\bullet$. 
\end{theorem}
\begin{proof}
Let $\frakA_\frakP^\wedge$ have a parametrix $\frakB$ in $L^\bullet$, made up of the operators
$B_{[j]}\in L^0(\frakg_{[j]}^{(-1)})=L^0((g^{[j+1]},g^{[j]}))$. Let us represent $B_{[j]}$ as a block-matrix, 
 $$B_{[j]}=
     \begin{pmatrix}
      B^{[j]}_{j+1,j} & B^{[j]}_{j,j} & B^{[j]}_{j-1,j} & \cdots & B^{[j]}_{0,j} \\
      \vdots & \vdots & \vdots &  & \vdots \\
      B^{[j]}_{j+1,0} & B^{[j]}_{j,0} & B^{[j]}_{j-1,0} & \cdots & B^{[j]}_{0,0}
     \end{pmatrix},
     \qquad B_{k,\ell}^{[j]}\in L^0((g^k,g^\ell)).
 $$
Since $\frakB$ is a parametrix to $\frakA_\frakP^\wedge$, we have 
\begin{equation}\label{eq:A}
 A_{[j-1]}B_{[j-1]}+B_{[j]}A_{[j]}=1+C_{[j]},\qquad C_{[j]}\in L^{-\infty}((g^{[j]},g^{[j]})).
\end{equation}
Similarly as before, let us write $C_{[j]}=\big(C^{[j]}_{k,\ell}\big)$ with 
$C^{[j]}_{k,\ell}\in L^{-\infty}((g^k,g^\ell))$. Inserting in \eqref{eq:A} the block-matrix representations 
and looking only to the upper left corners, we find that 
 $$A_{j-1}B_{j,j-1}^{[j-1]}+B_{j+1,j}^{[j]}A_j+B_{j,j}^{[j]}(1-P_j)=1+C_{j,j}^{[j]}.$$
Multiplying this equation from the left and the right with $P_j$ and defining 
 $$B_j:=P_jB_{j+1,j}^{[j]}P_{j+1}\in T^0((g^{j+1},g^j);P_{j+1},P_j)$$
we find 
 $$A_{j-1}B_{j-1}+B_jA_j-P_j\in T^{-\infty}((g^{j},g^j);P_{j},P_j),\qquad j=0,1,2,\ldots$$
Thus the sequence of the $B_j$ is a parametrix in $T^\bullet$ of $\frakA_\frakP$. 
\end{proof}

In case the parametrix of $\frakA_\frakP^\wedge$ is also a complex, the resulting parametrix for $\frakA_\frakP$ 
will, in general, not be a complex. We must leave it as an open question whether $($or under which conditions$)$ 
it is possible to find a parametrix of $\frakA_\frakP$ which is a complex. 

\begin{theorem}\label{thm:T-complex-main}
Let $L^\bullet$ have both the block-matrix property and the Fredholm property. 
For an at most semi-infinite complex $\frakA_\frakP$ in $T^\bullet$ as in \eqref{eq:T-complex2}, the following 
are equivalent$:$ 
\begin{itemize}
 \item[a$)$] $\frakA_\frakP$ is a Fredholm complex. 
 \item[b$)$]  $\frakA_\frakP$ has a parametrix in $T^\bullet$  $($in the sense of Definition $\mathrm{\ref{def:param_in_T}})$.
\end{itemize}
If $L^\bullet=L^\bullet_\cl$ is classical, these properties are equivalent to 
\begin{itemize}
 \item[c$)$] $\frakA_\frakP$ is an elliptic complex $($in the sense of Definition $\mathrm{\ref{def:ellipticity_in_T}})$. 
\end{itemize}
\end{theorem}
\begin{proof}
Clearly, b$)$ implies a$)$. 
If a$)$ holds, the lifted complex $\frakA_\frakP^\wedge$ is a Fredholm complex. According to Proposition 
\ref{prop:L-complex-param} it has a parametrix. By Theorem \ref{thm:lift-param} we thus find a parametrix in 
$T^\bullet$ of $\frakA_\frakP$. Thus a$)$ implies b$)$.

Now let $L^\bullet$ be classical. If $\frakA_\frakP^\wedge$ is the lifted complex, then the family of complexes 
$\sigma_\ell(\frakA_\frakP^\wedge)$ in the sense 
of \eqref{eq:symbol-complex} is the lift of the family of complexes $\sigma_\ell(\frakA_\frakP)$ from 
\eqref{eq:symbol-complex-toeplitz}. Thus, due to Theorem \ref{thm:lift-param} $($applied in each fibre$)$, 
$\frakA_\frakP$ is elliptic if, and only if, $\frakA_\frakP^\wedge$ is. 
By Theorem \ref{thm:L-complex-classical}, the latter is equivalent to the 
Fredholmness of $\frakA_\frakP^\wedge$ which, again by Theorem \ref{thm:lift-param}, is equivalent to the 
Fredholmness of $\frakA_\frakP$. This shows the equivalence of a$)$ and c$)$. 
\end{proof}

Now we generalize Theorem \ref{thm:L-lift1} to complexes in Toeplitz algebras. 

\begin{theorem}\label{thm:L-lift2}
Assume that $L^\bullet_\cl$ has both the block-matrix property and the extended Fredholm property. 
Let $N\in\nz$ and 
$A_j\in T^0(\frakg_j;P_j,P_{j+1})$, $\frakg_j=(g^j,g^{j+1})$, $j=0,\ldots, N$, 
be such that the associated sequences of principal symbols form exact families of complexes 
\begin{equation*}
 \;0\lra E_\ell(g^0,P_0)\xrightarrow{\sigma_\ell(A_0;P_0,P_1)}
 \ldots \xrightarrow{\sigma_\ell(A_N;P_N,P_{N+1})}E_\ell(g^{N+1},P_{N+1})\lra 0, 
\end{equation*}
$\ell=1,\ldots,n$. Then there exist operators $\wt{A}_j\in T^{0}(\frakg_j;P_j,P_{j+1})$, 
$j=0,\ldots,N$, with $\sigma(\wt{A}_j;P_j,P_{j+1})=\sigma(A_j;P_j,P_{j+1})$ and  such that 
 $$\wt{\frakA}_\frakP:\;0\lra H(g^0,P_0)\xrightarrow{\wt{A}_0}H(g^1,P_1)
     \ldots\xrightarrow{\wt{A}_N}H(g^{N+1},P_{N+1})\lra0$$
is a complex. In case $A_{j+1}A_{j}$ is smoothing for every $j$, the operators $\wt{A}_j$ can be chosen 
in such a way that $\wt{A}_j-A_j$ is smoothing for every $j$. 
\end{theorem}
\begin{proof}
Consider the finite complex as a semi-infinite one, i.e., for $j>N$ we let $g^j=g$ be the weight such that 
$H(g)=\{0\}$ and denote by $A_j$ be the zero operator acting in $\{0\}$. Then we let 
 $$A_{[j]}\in L^0(\frakg_{[j]}),\qquad \frakg_{[j]}:=(g^{[j]},g^{[j+1]}),$$
as defined in \eqref{eq:a_jlift}.  This defines a series of operators 
$A_{[0]},A_{[1]},A_{[2]},\ldots$ which, in general, is infinite, i.e., the operators $A_{[j]}$ with $j>N$ 
need not vanish. However, by construction, we have that 
\begin{equation}\label{eq:lift2}
 A_{[j+1]}A_{[j]}=0\qquad\forall\;j\ge N.
\end{equation}
Moreover, the associated families of complexes of principal symbols are exact families due to 
Proposition \ref{prop:lift}. We now modify the operator $A_{[N-1]}$ using the procedure described in the 
proof of Theorem \ref{thm:L-lift1} $($due to \eqref{eq:lift2} the operators $A_{[j]}$ with $j\ge N$ do not 
need to be modified$)$. 

Thus let $\Pi_{[N]}$ be the orthogonal projection in $H(g^{[N]})$ onto 
\begin{align*}
 \mathrm{ker}\,A_{[N]}=&  
 \mathrm{ker}\,\big(A_N:H(g^N,P_N)\to H(g^{N+1},P_{N+1})\big)\oplus \\
 &\oplus \mathrm{ker}\,P_{N-1}\oplus \mathrm{im}\,P_{N-2}\oplus \mathrm{ker}\,P_{N-3}\oplus\ldots
\end{align*}
and $\wt{A}_{[N-1]}:=\Pi_{[N]}A_{[N-1]}=A_{[N-1]}+R_{[N-1]}$ with $R_{[N-1]}\in L^0(\frakg_{[N-1]})$ 
having vanishing principal symbol. If we write $\Pi_{[N]}$ in block-matrix form, the entry 
$\Pi_{[N]}^{11}\in L^0(\frakg_N)$ in the upper left corner is the orthogonal projection of $H(g^N)$ onto 
$\mathrm{ker}\,\big(A_N:H(g^N,P_N)\to H(g^{N+1},P_{N+1})\big)$. Thus 
$\Pi_N:=P_N\Pi_{[N]}^{11}P_N\in T^0(\frakg_N;P_N,P_N)$ is the orthogonal projection of $H(g_N,P_N)$ onto 
the same kernel. Comparing the upper left corners of $A_{[N-1]}$ and $\wt{A}_{[N-1]}=\Pi_{[N]}A_{[N-1]}$ 
we find that 
 $$\Big(\Pi_{[N]}^{11}A_{N-1}+\Pi_{[N]}^{12}(1-P_{N-1})\Big)-A_{N-1}=R^{11}_{[N-1]}.$$ 
Multiplying by $P_N$ from the left and by $P_{N-1}$ from the right yields that 
$A_{N-1}$ differs from $\wt{A}_{N-1}:=\Pi_N A_{N-1}\in T^0(\frakg_N;P_N,P_N)$ by 
 $$R_{N-1}:=P_NR^{11}_{[N-1]}P_{N-1}\in T^0(\frakg_N;P_{N-1},P_N).$$ 
Moreover, $R_{N-1}$ has vanishing symbol $\sigma(R_{N-1};P_{N-1},P_N)$ and $A_N\wt{A}_{N-1}=0$, 
since $\Pi_N$ maps into the kernel of $A_N$. 

Now we replace $A_{N-1}$ by $\wt{A}_{N-1}$ and repeat this procedure to modify $A_{N-2}$, and so on 
until modification of $A_0$. 
\end{proof}

\subsection{Complexes on manifolds with boundary revisited}\label{sec:06.4}

Let us now apply our results to complexes on manifolds with boundary, i.e., to complexes in Boutet de Monvel's 
algebra and its APS version. In particular, we shall provide details we already have made use of in 
Section \ref{sec:03.2} on boundary value problems for complexes.  

In the following we work with the operators  
 $$\scrA_j\in\scrB^{\mu_j,d_j}(\Omega;(E_j,F_j;P_j),(E_{j+1},F_{j+1};P_{j+1})),\qquad j=0,\ldots,n.$$

\subsubsection{Complexes in Boutet's algebra with APS type conditions}\label{sec:06.4.1}

Assume $\scrA_{j+1}\scrA_j=0$ for every $j$. For convenience we introduce the notation 
 $$\bfH^s_j:=H^s(\Omega,E_j)\oplus H^{s}(\partial\Omega,F_j;P_j)$$
and the numbers $\nu_j:=\mu_0+\ldots+\mu_j$. 
Then we obtain finite complexes 
\begin{equation}\label{eq:complex_bdm_rev}
     \frakA_\frakP:0\lra \bfH^s_0\xrightarrow{\scrA_{0}}
     \bfH^{s-\nu_0}_1\xrightarrow{\scrA_1}
     \bfH^{s-\nu_1}_2\xrightarrow{\scrA_2}
     \ldots\xrightarrow{\scrA_n}
     \bfH^{s-\nu_{n}}_{n+1}\lra0
\end{equation}
for every integer $s\ge s_{\min}$  with 
 $$s_{\min}:=\max\big\{\nu_j,d_j+\nu_{j-1}\mid j=0,\ldots,n\big\} \qquad (\text{with $\nu_{-1}:=0$}).$$

The complex $\frakA_\frakP$ is called elliptic if both associated families of complexes $\sigma_\psi(\frakA_\frakP)$ and 
$\sigma_\partial(\frakA_\frakP)$, made up of the associated principal symbols and principal boundary symbols, 
respectively, are exact. In fact, ellipticity is independent of the index $s$. 

\begin{theorem}\label{thm:complex_bdm_rev}
The following statements are equivalent$:$ 
\begin{itemize}
 \item[a$)$]  $\frakA_\frakP$ is elliptic. 
 \item[b$)$] $\frakA_\frakP$ is a Fredholm complex for some $s\ge s_{\min}$. 
 \item[c$)$] $\frakA_\frakP$ is a Fredholm complex for all $s\ge s_{\min}$. 
\end{itemize}
In this case, $\frakA_\frakP$ has a parametrix made up of operators belonging to the APS-version of  
Boutet de Monvel's algebra. Moreover, the index of the complex does not depend on $s$.  
\end{theorem}
\begin{proof}
We shall make use of order reductions 
 $$\scrR^m_j:=
     \begin{pmatrix}\Lambda^m_j&0\\ 0& \lambda_j^m\end{pmatrix}
     \;\in\;\calB^{m,0}(\Omega;(E_j,F_j),(E_{j},F_{j})),$$
as already described in the discussion following Theorem \ref{thm:complex-main}, i.e., $\scrR_j^m$ 
is invertible with inverse given by $\scrR_j^{-m}$. 

Let $\frakA$ be elliptic. Define 
 $$\scrA^\prime_j=\scrR^{s_{\min}-\nu_j}_{j+1}\,\scrA_j\,\scrR^{\nu_{j-1}-s_{\min}}_{j},
     \qquad 
     P^\prime_j=\lambda_j^{s_{\min}-\nu_{j-1}}\,P_j\,\lambda_j^{\nu_{j-1}-s_{\min}}.$$
Then 
 $$\scrA^\prime_j\;\in\;\calB^{0,0}(\Omega;(E_j,F_j;P^\prime_j),(E_{j+1},F_{j+1};P^\prime_{j+1}))$$
and $\scrA^\prime_{j+1}\scrA^\prime_j=0$ for every $j$, i.e., the $\scrA^\prime_j$ induce a complex 
$\frakA^\prime_{\frakP^\prime}$ in the respective $L^2$-spaces, which remains elliptic.  
By Theorem \ref{thm:T-complex-main} $($with $L^\bullet_\cl=\calB^{\bullet,0}$ as described in Example 
\ref{ex:algebra-ex2}$)$ there exists a parametrix of $\frakA^\prime_{\frakP^\prime}$, made up by operators 
 $$\scrB^\prime_j\;\in\;\calB^{0,0}(\Omega;(E_{j+1},F_{j+1};P^\prime_{j+1}),(E_j,F_j;P^\prime_j)).$$
Then   
 $$\scrB_j:=\scrR^{\nu_{j-1}-s_{\min}}_{j}\,\scrB^\prime_j\,\scrR^{s_{\min}-\nu_j}_{j+1}\;\in\;
     \calB^{0,e_j}(\Omega;(E_{j+1},F_{j+1};P_{j+1}),(E_j,F_j;P_j)),$$
with $e_j:=s_{\min}-\nu_j$ and we obtain that 
 $$\scrA_{j-1}\scrB_{j-1}+\scrB_{j}\scrA_j-1\;\in\;
      \calB^{-\infty,e_j}(\Omega;(E_{j},F_{j};P_{j}),(E_j,F_j;P_j)).$$
Thus the induced operators $\scrB_{j}:\bfH^{s-\nu_j}_{j+1}\to\bfH^{s-\nu_{j-1}}_{j}$ give a parametrix 
of \eqref{eq:complex_bdm_rev} whenever $s\ge s_{\min}$. Summing up, we have verified that a$)$ implies 
c$)$. 

Now assume that b$)$ holds for one $s=s_0$. Similarly as before, we pass to a new Fredholm complex 
$\frakA^\prime_{\frakP^\prime}$ made up by the operators 
$\scrA^\prime_j=\scrR^{s_{0}-\nu_j}_{j+1}\,\scrA_j\,\scrR^{\nu_{j-1}-s_{0}}_{j}$, which have order 
and type $0$. By Theorem \ref{thm:T-complex-main} this complex is elliptic, and hence also the original complex 
$\frakA_\frakP$ is. Hence a$)$ holds. 

It remains to verify the indepence of $s$ of the index. However, this follows from the fact that the index of 
$\frakA_\frakP$ coincides with the index of its lifted complex $\frakA_\frakP^\wedge$ $($cf.\ Proposition 
\ref{prop:lift}$)$. The index of the latter is known to be independent of $s$, see for instance Theorem 2 
on page 283 of \cite{ReSc}.  
\end{proof}

\subsubsection{From principal symbol complexes to complexes of operators}\label{sec:06.4.2}

Theorem \ref{thm:L-lift2} in the present situation takes the following form$:$

\begin{theorem}\label{thm:lift-bdm-rev}
Assume that both the sequence of principal symbols $\sigma_\psi^{\mu_j}(\scrA_j)$ and the sequence  
of principal boundary symbols $\sigma^{\mu_j}_\psi(\scrA_j;P_j,P_{j+1})$ induce exact families of complexes. 
Then there exist operators 
 $$\tilde\scrA_j\in\scrB^{\mu_j,s_{\min}-\nu_{j-1}}(\Omega;(E_j,F_j;P_j),(E_{j+1},F_{j+1};P_{j+1})),\qquad 
     j=0,\ldots,n,$$
with $\tilde\scrA_{j+1}\tilde\scrA_j=0$ for every $j$ and such that   
 $$\scrA_j-\tilde\scrA_j\in\scrB^{\mu_j-1,s_{\min}-\nu_{j-1}}(\Omega;(E_j,F_j;P_j),(E_{j+1},F_{j+1};P_{j+1})).$$ 
\end{theorem}
\begin{proof}
Define operators $\scrA^\prime_j$ as in the beginning of the proof of Theorem \ref{thm:complex_bdm_rev}. 
These have order and type $0$ and satisfy the assumptions of the Theorem. Then by 
Theorem \ref{thm:L-lift2} there exist 
 $${\tilde\scrA}^\prime_j\in\scrB^{0,0}(\Omega;(E_j,F_j;P^\prime_j),(E_{j+1},F_{j+1};P^\prime_{j+1}))$$
with ${\tilde\scrA}^\prime_{j+1}{\tilde\scrA}^\prime_j=0$ and 
 $$\scrA^\prime_j-{\tilde\scrA}^\prime_j\in\scrB^{-1,0}(\Omega;(E_j,F_j;P^\prime_j),
     (E_{j+1},F_{j+1};P^\prime_{j+1})).$$ 
By choosing 
$\tilde\scrA_j:=\scrR^{\nu_j-s_{\min}}_{j+1}\,{\tilde\scrA}^\prime_j\,\scrR^{s_{\min}-\nu_{j-1}}_{j}$, 
the claim follows. 
\end{proof}

We conclude this section with a particular variant of Theorem \ref{thm:lift-bdm-rev}, 
which we need for completing the proof of Theorem \ref{thm:complex-main-order-zero}. 

\begin{proposition}\label{prop:lift_bdm_rev}
Let the $\scrA_j$ be as in Theorem $\mathrm{\ref{thm:lift-bdm-rev}}$ of order and type $0$. Furthermore,  
assume that 
$\scrA_j=\begin{pmatrix}A_j&K_j\\0&Q_j\end{pmatrix}$ and that $A_{j+1}A_j=0$ for every $j$. Then the 
$\tilde\scrA_j$ from Theorem $\mathrm{\ref{thm:lift-bdm-rev}}$ can be chosen in the form 
$\tilde\scrA_j=\begin{pmatrix}A_j&\tilde{K}_j\\0&\tilde{Q}_j\end{pmatrix}$. 
\end{proposition}
\begin{proof}
To prove this result we recall from the proof of Theorem \ref{thm:L-lift2} that the $\tilde\scrA_j$ are constructed 
by means of an iterative procedure, choosing $\tilde\scrA_n:=\scrA_n$ and then modifying 
$\scrA_{n-1},\ldots,\scrA_0$ one after the other. 
In fact, if $\tilde\scrA_n,\ldots,\tilde\scrA_{k+1}$ are constructed and have the form as stated, then 
$\tilde\scrA_k:=\pi_{k+1}\scrA_k$ with $\pi_{k+1}$ being the orthogonal projection in 
$L^2(\Omega,E_{k+1})\oplus L^2(\pO,F_{k+1};P_{k+1})$ onto the kernel of $\tilde\scrA_{k+1}$. 
Now let $u\in L^2(\Omega,E_k)$ be arbitrary. Since $A_{k+1}A_k=0$, it follows that $(A_ku,0)$ belongs to 
$\mathrm{ker}\,\tilde\scrA_{k+1}$ and thus 
 $$\tilde\scrA_k\begin{pmatrix}u\\0\end{pmatrix}=\pi_{k+1}\begin{pmatrix}A_ku\\0\end{pmatrix}=
     \begin{pmatrix}A_ku\\0\end{pmatrix}.$$
Hence the block-matrix representation of $\tilde\scrA_k$ has the desired form.  
\end{proof}

\bibliographystyle{amsalpha}

\end{document}